	\documentclass[reqno]{amsart}  
\usepackage{amsmath}
  \usepackage{paralist}
  \usepackage{graphics} 
  \usepackage{epsfig} 
\usepackage{graphicx}  \usepackage{epstopdf}
\usepackage{amssymb,amsfonts,amstext,amsthm}
\usepackage{enumerate}
\usepackage{subfigure}

\newtheorem{theorem}{Theorem}

\newtheorem{lemma}[theorem]{Lemma}

\theoremstyle{definition}

\newtheorem{remark}[theorem]{Remark}

\title[A loop type component in the non-negative solutions set] 
      {A loop type component in the non-negative solutions set of an indefinite elliptic problem}

\author[Humberto Ramos Quoirin and Kenichiro Umezu]{}


 \email{humberto.ramos@usach.cl}
 \email{kenichiro.umezu.math@vc.ibaraki.ac.jp}

\thanks{The first author was supported by the FONDECYT grants 1161635, 1171532 and 1171691.}
\thanks{The second author was supported by JSPS KAKENHI Grant Number 15K04945.}

\begin{document}

\maketitle

\centerline{\scshape Humberto Ramos Quoirin}
\medskip
{\footnotesize
 \centerline{Universidad de Santiago de Chile }
   \centerline{Casilla 307, Correo 2, Santiago, Chile}
} 

\medskip

\centerline{\scshape Kenichiro Umezu}
\medskip
{\footnotesize
 \centerline{Department of Mathematics, Faculty of Education}
 \centerline{Ibaraki University, Mito 310-8512, Japan}
}

\bigskip


\begin{abstract}
We prove the existence of a loop type component of non-negative solutions for 
an indefinite elliptic equation with a homogeneous Neumann boundary condition. 
This result complements our previous results obtained in \cite{RQU}, where the 
existence of another loop type component was established in a different 
situation. Our proof combines local and global bifurcation theory, rescaling 
and regularizing arguments, a priori bounds, and Whyburn's topological method. 
A further investigation of the loop type component established in \cite{RQU} 
is also provided.
\end{abstract}


\section{Introduction}

Let $\Omega$ be a smooth bounded domain of $\mathbb{R}^N$,
$N\geq 1$. This article is devoted to the problem
\[
\begin{cases}
-\Delta u = \lambda b(x) u^{q-1} + a(x)u^{p-1} & \mbox{in $\Omega$}, \\
\dfrac{\partial u}{\partial \mathbf{n}} = 0 & \mbox{on $\partial \Omega$},
\end{cases} \leqno{(P_\lambda)}
\]
where
\begin{itemize}
\item $\lambda \in \mathbb{R}$;
\item  $1<q<2<p$;
\item  $a,b \in C^\alpha (\overline{\Omega})$, $\alpha \in (0,1)$;
\item  $a\not\equiv 0$ and $b$ changes sign;
\item $\mathbf{n}$ is the unit outer normal to $\partial \Omega$.
\end{itemize}

By a solution of $(P_\lambda)$, we mean a classical solution. A solution $u$ of $(P_\lambda)$ is said to be {\it nontrivial and non-negative} if it satisfies $u\geq 0$ on $\overline{\Omega}$ and $u\not\equiv 0$, whereas it is said to be  {\it positive} if it satisfies $u>0$ on $\overline{\Omega}$. Note that since $b$ changes sign and $1<q<2$, the strong maximum principle does not apply and, as a consequence, we can not deduce that nontrivial non-negative solutions of $(P_\lambda)$ are actually positive solutions,
unlike in the case $q\geq 2$.

In \cite{RQU} we have investigated existence, non-existence, and multiplicity of non-negative solutions as well as their asymptotic behavior as $\lambda \to 0$. These results led us to analyse the structure of the set of non-negative solutions of $(P_\lambda)$. In particular, we have proved the existence of a loop type component in this set, under the following conditions (see \cite[Theorem 1.6]{RQU} and Figure \ref{fig16_0730f}):
\begin{align} \label{prehypo}
a,b \mbox{ change sign,} \quad \int_\Omega b \leq 0, \quad \mbox{and}
\quad \int_\Omega a <0.
\end{align}
We shall assume that $a$ and $b$ are positive in some open ball (see $(H_0)$ below), in which case the nonlinearity in $(P_\lambda)$ has (locally) a {\it concave-convex} nature. We refer the reader to \cite{RQU} for a more general discussion on $(P_\lambda)$ and related concave-convex problems.

Our purpose is to go further in this investigation, focusing now mostly on the case
\begin{equation}\label{160730hypo}
\int_\Omega b < 0 \leq \int_\Omega a.
\end{equation}
Before stating our result, let us set
\begin{align*}
\Omega^a_\pm = \{ x \in \Omega : a\gtrless 0 \}, \quad
\Omega^b_\pm = \{ x \in \Omega : b\gtrless 0 \}.
\end{align*}

The following conditions shall be assumed in our main result:
\begin{enumerate}

\item[($H_0$)] $a(x_0), b(x_0)>0$ for some $x_0 \in \Omega$;

\item[($H_1$)] $\Omega^a_+$ and $\Omega' := \Omega \setminus \overline{\Omega^a_+}$ are subdomains of $\Omega$ with smooth boundaries, and satisfy either $\overline{\Omega^a_+}\subset \Omega$ or $\overline{\Omega'} \subset \Omega$;

\item[($H_2$)] There exist $\gamma > 0$ and a function $\alpha^+$ which is continuous, positive, and bounded away from zero in a tubular neighborhood $U$ of $\partial \Omega^a_+$ in $\Omega^a_+$, such that
\begin{align*}
& a^+(x) = \alpha^+(x) \, {\rm dist}(x, \partial \Omega^a_+)^\gamma, \quad x \in U, \\
& 2 < p < \min \left\{ \frac{2N}{N-2}, \ \frac{2N+\gamma}{N-1} \right\} \quad \mbox{if} \quad N>2;
\end{align*}
\item[($H_3$)] $\Omega^b_\pm$ are subdomains of $\Omega$.
\end{enumerate}

These conditions guarantee some {\it a priori} bounds in $(0,\infty) \times C(\overline{\Omega})$ for non-negative solutions. More precisely, $(H_0)$ implies that $(Q_{\mu, \epsilon})$, a rescaled and regularized version of
$(P_\lambda)$, has no positive solutions for $\mu$ sufficiently large,
similarly as \cite[Proposition 6.1]{RQU}. On the other hand, $(H_1)$ and $(H_2)$ provide us with an {\it a priori} bound on $\|u\|_{C(\overline{\Omega})}$ for any non-negative solution $u$ of $(Q_{\mu, \epsilon})$,
see \cite[Proposition 6.5]{RQU}. Let us mention that $(H_2)$ goes back to Amann and  L\'{o}pez-G\'{o}mez \cite{ALG}, where the authors have established {\it a priori} bounds for positive solutions of indefinite elliptic problems. Finally, $(H_3)$ is employed to show that bifurcation from zero for nontrivial non-negative solutions of $(P_\lambda)$ does not occur at $\lambda \neq 0$,
as in \cite[Proposition 3.3]{RQU}.

We state now our main result,
which gives a positive answer to the open question raised in Subsection 6.1 of \cite{RQU}. 	
Note that, in contrast with \cite[Theorem 1.6]{RQU}, $a$ may be non-negative.

\begin{theorem}  \label{tp}
Assume \eqref{160730hypo}, $(H_0)$ and $(H_3)$. In addition, assume one of the following conditions:
\begin{enumerate}
  \item[(a)] $a > 0$ on $\overline{\Omega}$, and $2<p<\frac{2N}{N-1}$ if $N>2$;
  \item[(b)] $(H_1)$ and $(H_2)$ hold, and $(\Omega' \setminus \Omega^a_-) \subset \Omega^b_+$ if $\Omega' \neq \Omega^a_-$.
\end{enumerate}
Then $(P_\lambda)$ has a bounded
{\rm component} (non-empty, closed, and connected subset in $\mathbb{R}\times C(\overline{\Omega})$) of non-negative solutions $\mathcal{C}_0 = \{ (\lambda, u) \}$. In addition, $\mathcal{C}_0$ is of loop type, i.e., it is a bounded component that meets a single point on a trivial solution line and
joins this point to itself. More precisely, $\mathcal{C}_0$ starts and ends at $(0,0)$, and  has the following properties (see Figure \ref{fig16_0730e}):
\begin{enumerate}
  \item $\mathcal{C}_0 \cap \{ (\lambda, 0) : \lambda \neq 0 \} = \emptyset$.
Consequently, $\mathcal{C}_0 \setminus \{ (0,0) \}$ consists of nontrivial non-negative solutions.

  \item $(P_\lambda)$ has no nontrivial non-negative solution for $\lambda=0$, so that $u\equiv 0$ if $(0, u) \in \mathcal{C}_0$.

  \item There is no $(\lambda,u) \in \mathcal{C}_0$ with $\lambda<0$, i.e., $\mathcal{C}_0$ bifurcates to the region $\lambda > 0$ at $(0,0)$.

  \item There exist at least two nontrivial non-negative solutions $(\lambda,u_{1,\lambda}), (\lambda,u_{2,\lambda}) \in \mathcal{C}_0$, for $\lambda > 0$ sufficiently small.
\end{enumerate}
\end{theorem}

\begin{remark}
\strut
\begin{enumerate}
  \item Condition (a) can be understood as $(H_2)$ with $\gamma = 0$.
  \item Condition (b) includes the following cases:
\begin{enumerate}
  \item[(1)]
$a(x)>0$ in $\Omega$, and $a(x)$ vanishes on $\partial \Omega$. This situation is understood as $\Omega' = \emptyset$.
\item[(2)] $\{ x\in \Omega' : a(x)=0 \}\neq \emptyset$. In particular, it includes the case that $a(x)$ changes sign, as well as the case that $a(x) \geq 0$ in $\Omega$. In both cases we need that $b(x)>0$ in $\{ x\in \Omega' : a(x)=0 \}$;

\end{enumerate}
\end{enumerate}
\end{remark}

\begin{remark}
The existence of the loop type component $\mathcal{C}_0$ provided by Theorem \ref{tp} is consistent with \cite[Theorem 1.1]{RQU}. As a matter of fact, in \cite[Theorem 1.1]{RQU} it is proved that if \eqref{160730hypo} holds then $(P_\lambda)$ has two nontrivial non-negative solutions for $\lambda>0$ sufficiently small. Moreover, these solutions converge both to $0$ in $C^2(\overline{\Omega})$ as $\lambda \to 0^+$. We believe that these solutions correspond to the upper and lower branches of $\mathcal{C}_0$.
\end{remark}

	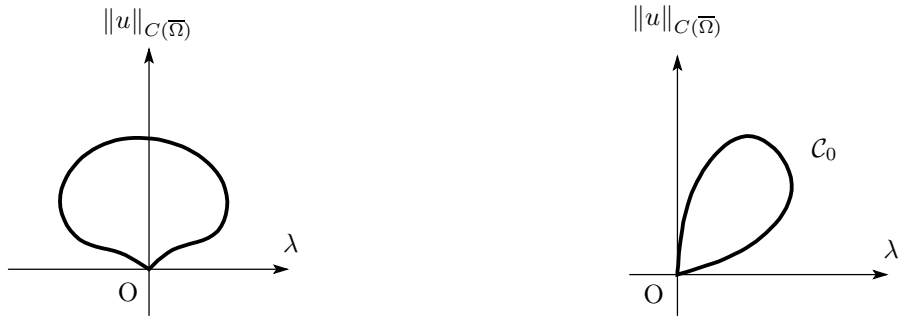
\begin{figure}[!htb]
      \begin{center}
      \subfigure[Minimal possibility for $\mathcal{C}_0$ when $a,b$ change sign, $\int_\Omega b \leq 0$, and $\int_\Omega a < 0$.]{
    %
{\unitlength 0.1in
\begin{picture}( 17.6900, 16.0300)( 21.5900,-23.1000)
%
\special{pn 8}%
\special{pa 2470 2068}%
\special{pa 3928 2068}%
\special{fp}%
\special{sh 1}%
\special{pa 3928 2068}%
\special{pa 3862 2048}%
\special{pa 3876 2068}%
\special{pa 3862 2088}%
\special{pa 3928 2068}%
\special{fp}%
%
\special{pn 8}%
\special{pa 3206 2310}%
\special{pa 3206 916}%
\special{fp}%
\special{sh 1}%
\special{pa 3206 916}%
\special{pa 3186 984}%
\special{pa 3206 970}%
\special{pa 3226 984}%
\special{pa 3206 916}%
\special{fp}%
\put(30.9700,-21.8400){\makebox(0,0){O}}%
\put(39.4800,-19.2600){\makebox(0,0){$\lambda$}}%
\put(31.9900,-7.8700){\makebox(0,0){$\| u \|_{C(\overline{\Omega})}$}}%
%
\special{pn 20}%
\special{pa 3206 2068}%
\special{pa 3252 2022}%
\special{pa 3304 1982}%
\special{pa 3330 1966}%
\special{pa 3360 1954}%
\special{pa 3424 1934}%
\special{pa 3456 1926}%
\special{pa 3488 1916}%
\special{pa 3516 1904}%
\special{pa 3542 1888}%
\special{pa 3566 1868}%
\special{pa 3584 1842}%
\special{pa 3598 1814}%
\special{pa 3608 1782}%
\special{pa 3614 1748}%
\special{pa 3616 1714}%
\special{pa 3612 1680}%
\special{pa 3604 1646}%
\special{pa 3594 1616}%
\special{pa 3578 1586}%
\special{pa 3560 1560}%
\special{pa 3538 1534}%
\special{pa 3516 1512}%
\special{pa 3490 1490}%
\special{pa 3464 1470}%
\special{pa 3408 1438}%
\special{pa 3378 1424}%
\special{pa 3318 1404}%
\special{pa 3288 1396}%
\special{pa 3256 1388}%
\special{pa 3224 1384}%
\special{pa 3160 1380}%
\special{pa 3128 1380}%
\special{pa 3096 1382}%
\special{pa 3064 1386}%
\special{pa 3032 1392}%
\special{pa 2972 1412}%
\special{pa 2942 1426}%
\special{pa 2912 1442}%
\special{pa 2886 1460}%
\special{pa 2858 1480}%
\special{pa 2834 1504}%
\special{pa 2812 1530}%
\special{pa 2790 1558}%
\special{pa 2774 1586}%
\special{pa 2760 1616}%
\special{pa 2748 1648}%
\special{pa 2742 1680}%
\special{pa 2740 1712}%
\special{pa 2742 1744}%
\special{pa 2750 1778}%
\special{pa 2762 1808}%
\special{pa 2776 1838}%
\special{pa 2796 1866}%
\special{pa 2818 1890}%
\special{pa 2842 1912}%
\special{pa 2870 1928}%
\special{pa 2900 1940}%
\special{pa 2930 1950}%
\special{pa 2962 1958}%
\special{pa 2992 1964}%
\special{pa 3024 1972}%
\special{pa 3054 1982}%
\special{pa 3084 1994}%
\special{pa 3140 2026}%
\special{pa 3194 2060}%
\special{pa 3206 2068}%
\special{fp}%
\end{picture}}%
          \label{fig16_0730f}
      }
      \hfill
      \subfigure[Minimal possibility for $\mathcal{C}_0$ when $a\not\equiv 0$, $b$ changes sign, $\int_\Omega b < 0$, and $\int_\Omega a \geq 0$.]{
{\unitlength 0.1in
\begin{picture}( 21.3500, 16.2200)( 14.3900,-22.1000)
%
\special{pn 8}%
\special{pa 2230 1996}%
\special{pa 3574 1996}%
\special{fp}%
\special{sh 1}%
\special{pa 3574 1996}%
\special{pa 3508 1976}%
\special{pa 3522 1996}%
\special{pa 3508 2016}%
\special{pa 3574 1996}%
\special{fp}%
%
\special{pn 8}%
\special{pa 2480 2210}%
\special{pa 2480 854}%
\special{fp}%
\special{sh 1}%
\special{pa 2480 854}%
\special{pa 2460 922}%
\special{pa 2480 908}%
\special{pa 2500 922}%
\special{pa 2480 854}%
\special{fp}%
\put(23.5400,-20.9600){\makebox(0,0){O}}%
\put(35.9800,-18.5800){\makebox(0,0){$\lambda$}}%
\put(24.7900,-6.6800){\makebox(0,0){$\| u \|_{C(\overline{\Omega})}$}}%
%
\special{pn 20}%
\special{pa 2480 1996}%
\special{pa 2542 1978}%
\special{pa 2632 1948}%
\special{pa 2664 1938}%
\special{pa 2694 1928}%
\special{pa 2722 1916}%
\special{pa 2752 1902}%
\special{pa 2780 1888}%
\special{pa 2810 1874}%
\special{pa 2838 1858}%
\special{pa 2892 1822}%
\special{pa 2944 1778}%
\special{pa 2992 1730}%
\special{pa 3014 1704}%
\special{pa 3032 1676}%
\special{pa 3050 1646}%
\special{pa 3062 1616}%
\special{pa 3072 1586}%
\special{pa 3078 1554}%
\special{pa 3078 1522}%
\special{pa 3074 1488}%
\special{pa 3064 1456}%
\special{pa 3050 1422}%
\special{pa 3032 1392}%
\special{pa 3012 1362}%
\special{pa 2988 1336}%
\special{pa 2960 1314}%
\special{pa 2932 1294}%
\special{pa 2902 1280}%
\special{pa 2872 1272}%
\special{pa 2842 1270}%
\special{pa 2812 1276}%
\special{pa 2784 1286}%
\special{pa 2756 1300}%
\special{pa 2728 1320}%
\special{pa 2702 1344}%
\special{pa 2678 1368}%
\special{pa 2654 1396}%
\special{pa 2614 1452}%
\special{pa 2598 1482}%
\special{pa 2582 1510}%
\special{pa 2554 1570}%
\special{pa 2544 1600}%
\special{pa 2534 1632}%
\special{pa 2518 1692}%
\special{pa 2510 1724}%
\special{pa 2506 1756}%
\special{pa 2496 1820}%
\special{pa 2492 1850}%
\special{pa 2488 1882}%
\special{pa 2486 1914}%
\special{pa 2484 1948}%
\special{pa 2480 1980}%
\special{pa 2480 1996}%
\special{fp}%
\put(32.5000,-13.4000){\makebox(0,0){$\mathcal{C}_0$}}%
\end{picture}}%
        \label{fig16_0730e}
      }
      \end{center}
      \caption{Loop type components of nontrivial non-negative solutions of $(P_\lambda)$.}
      \label{fig15_111402}
    \end{figure}
%
%
%
The existence  
of bounded (or compact) components in the solution set of nonlinear problems has been investigated by Cingolani and G\'amez \cite{CD96}, Cano-Casanova \cite{Ca04}, L\'opez-G\'omez and Molina-Meyer \cite{LGMM}, and Brown \cite{B07}. A {\it mushroom}, i.e. a component connecting two simple eigenvalues of the linearized eigenvalue problem at the trivial solution $u=0$,  was obtained by Cingolani and G\'amez for both the Dirichlet case and $\Omega = \mathbb{R}^N$, and by Cano-Casanova for a mixed linear boundary condition. In addition to the existence of a mushroom, L\'opez-G\'omez and Molina-Meyer (for the Dirichlet case) and Brown (for the Neumann case) obtained a {\it loop}, i.e. a component that meets a single point on the trivial solution line. Moreover, L\'opez-G\'omez and Molina-Meyer also proved the existence of an  {\it isola}, i.e. a component that does not touch the trivial solution line.
Finally, we refer to \cite[Theorem 1.6]{RQU}, where the existence of a loop type component for $(P_\lambda)$ was proved in case \eqref{prehypo}.

Let us remark that the nonlinearities in \cite{CD96, Ca04, LGMM, B07} are $C^1$ at $u=0$, which is not the case for $(P_\lambda)$. Therefore the standard global bifurcation theory of Rabinowitz \cite{Ra71} (see also L\'opez-G\'omez \cite{LG01}) does not apply straightforwardly to $(P_\lambda)$.
To overcome this difficulty, we employ a {\it regularization}  method around the trivial solution and develop Whyburn's topological analysis
\cite[(9.12)Theorem]{Wh64} to convert the bifurcation results obtained for the regularized problem to the original problem.
In fact, before considering the regularization, we carry out a {\it scaling} argument to overcome a difficulty which appears in case \eqref{160730hypo}. Unlike in case  \eqref{prehypo}, it is difficult to study directly $(P_\lambda)$ and its
regularization under \eqref{160730hypo}, since these problems have no positive solutions for $\lambda = 0$ \cite[Lemma 6.8(1)]{RQU}. Even if we can prove the existence of a component of positive solutions for the regularized problem, the non-existence result for $\lambda=0$ may cause the shrinking of the component into the set of trivial solutions when the topological method is employed. It should be emphasized that in order to obtain the loop in case \eqref{prehypo} as in Figure
\ref{fig16_0730f}, the following fact was crucial: a component of positive solutions for the $\epsilon$-regularized problem of $(P_\lambda)$ {\it does cut} the vertical axis $\lambda = 0$, at some point that {\it does not shrink} to $(0,0)$ as $\epsilon \to 0$.


In order to verify that a component of non-negative solutions of $(P_\lambda)$ is bounded in $(0,\infty) \times C(\overline{\Omega})$, we shall make good use of {\it a priori} bounds for non-negative solutions of $(P_\lambda)$, as well as for non-negative solutions of $(Q_\mu)$ and $(Q_{\mu, \epsilon})$ below. We obtain these {\it a priori} bounds under either conditions (a) or (b) in Theorem \ref{tp}, proceeding in the same way
just as in \cite[Proposition 6.5]{RQU}.

The rest 
of this article is organized as follows. In Section \ref{sec:sr}, by the change of variables $\mu = \lambda^{\frac{p-2}{p-q}}$ and $v=\lambda^{-\frac{1}{p-q}} u$, we transform $(P_\lambda)$ into $(Q_\mu)$, and consider a $\epsilon$-regularized version of $(Q_\mu)$, i.e. $(Q_{\mu, \epsilon})$. This regularization scheme enables us to apply the local and global bifurcation theory from simple eigenvalues. We deduce then the existence of a  component of bifurcating positive solutions of $(Q_{\mu, \epsilon})$ from $\{ (\mu, 0)\}$. Section \ref{sec:proof} is devoted to the proof of our main result, Theorem \ref{tp}. Using Whyburn's topological method, we establish the limiting behavior of the component of $(Q_{\mu, \epsilon})$ obtained in Section \ref{sec:sr} as $\epsilon \to 0^+$, and obtain a component of nontrivial nonnegative solutions of $(Q_\mu)$ which bifurcates from $(0,0)$ into the region $\mu > 0$. Finally, by the scaling, we go back to $(P_\lambda)$, and obtain thus a bounded component of nontrivial nonnegative solutions which is of loop type, joins $(0,0)$ to itself, and lies in the region $\lambda > 0$, as shown in Figure \ref{fig16_0730e}. In Section \ref{sec:a<0}, we carry out a further analysis for the loop of nontrivial nonnegative solutions of $(P_\lambda)$ obtained in the case \eqref{prehypo} by \cite{RQU}. The analysis concentrates on the direction of the bifurcation point from which the loop emanates. The main result of this section is Theorem \ref{thm:dir}. \newline

\section{Scaling and regularization schemes}  \label{sec:sr}

We set $v = \lambda^{-\frac{1}{p-q}}u$ and $\mu = \lambda^{\frac{p-2}{p-q}}$ for $\lambda > 0$, so that $(P_\lambda)$ is transformed into
\[
\begin{cases}
-\Delta v = \mu \left(
b(x) v^{q-1} + a(x) v^{p-1} \right) & \mbox{in $\Omega$}, \\
\dfrac{\partial v}{\partial \mathbf{n}} = 0 & \mbox{on $\partial \Omega$},
\end{cases} \leqno{(Q_\mu)}
\]
where $\mu \geq 0$.
We note that the nonlinearity in $(Q_\mu)$ is {\bf not} differentiable at $v=0$, so that the local and global bifurcation theory from simple eigenvalues on the trivial line
\[
\Gamma_0 = \{ (\mu, 0) : \mu \geq 0\}.
\]
can not be directly applied  to $(Q_\mu)$. To overcome this difficulty, we shall consider the following regularized version of $(Q_\mu)$, where $\epsilon\in (0,1]$ is fixed:
\[
\begin{cases}
-\Delta v = \mu \left(
b(x) (v+\epsilon)^{q-2}v + a(x) v^{p-1} \right) & \mbox{in $\Omega$}, \\
\dfrac{\partial v}{\partial \mathbf{n}} = 0 & \mbox{on $\partial \Omega$}.
\end{cases} \leqno{(Q_{\mu, \epsilon})}
\]
It is understood that $(Q_{\mu, \epsilon}) = (Q_\mu)$ when $\epsilon = 0$. It is clear that, in addition to $\Gamma_{0}$, $(Q_{\mu, \epsilon})$ has the trivial line of positive solutions
\[
\Gamma_{00} = \{ (0, c) : c \ \mbox{is a non-negative constant} \}.
\]
Furthermore, by the strong maximum principle and the boundary point lemma, any nontrivial non-negative solution of $(Q_{\mu, \epsilon})$ is positive.

First, we discuss bifurcation from $\Gamma_{00}$.
\begin{lemma} \label{lem:bf00}
Let $\epsilon \in (0,1]$. Under the conditions of Theorem \ref{tp}, we have
the following:
\begin{enumerate}
\item Assume that $\int_\Omega a > 0$. Let $(\mu_n, v_n)$ be positive solutions of $(Q_{\mu_n, \epsilon})$ with  $\mu_n > 0$ and $(\mu_n, v_n)$ converging to $(0,c)$
in $\mathbb{R}\times C(\overline{\Omega})$ for some positive constant $c$.
Then $c=c^*_\epsilon$, where $c^*_\epsilon$ is the unique positive solution of the equation
\begin{align} \label{eq:c}
c^{p-2}(c+\epsilon)^{2-q} = \frac{-\int_\Omega b}{\int_\Omega a}.
\end{align}
Moreover,
\begin{align} \label{cepto0}
c^*_\epsilon \longrightarrow c^*_0 \quad \mbox{ as } \epsilon \to 0^+.
\end{align}
\item Assume that $\int_\Omega a = 0$. Then, there are no positive solutions $v_n$ of $(Q_{\mu_n, \epsilon})$ such that $\mu_n \to 0^+$, and $v_n \to c$ in $C(\overline{\Omega})$ for some constant $c > 0$.
\end{enumerate}
\end{lemma}

\begin{proof}
\strut
\begin{enumerate}
\item The divergence theorem shows that
\begin{align} \label{divvn}
0 = \frac{1}{\mu_n} \int_{\partial \Omega} \left( -\frac{\partial v_n}{\partial \mathbf{n}} \right)
= \frac{1}{\mu_n} \int_\Omega \left( -\Delta v_n\right) =
\int_\Omega \left\{ b(v_n + \epsilon)^{q-2} v_n + a v_n^{p-1} \right\}.
\end{align}
By passing to the limit as $n\to \infty$, it follows that $c$ satisfies \eqref{eq:c}.
Thus, $c=c^*_0$.  Moreover, assertion \eqref{cepto0} is verified in a trivial way.\\
\item If not, then, in the same way as in \eqref{divvn}, we deduce that
$0= c(c+\epsilon)^{q-2}\int_\Omega b < 0$, a contradiction.
\end{enumerate}
\end{proof}

\begin{remark} \label{rem:lem4}
The assertions of Lemma \ref{lem:bf00} except \eqref{cepto0} are also valid for $\epsilon = 0$.
\end{remark}

Now, using \cite[Theorem 1.7]{CR71}, we carry out a local bifurcation analysis for $(Q_{\mu, \epsilon})$ with $\epsilon > 0$ on $\Gamma_0$, where $\mu$ is the bifurcation parameter. To this end, we reduce $(Q_{\mu, \epsilon})$ to an operator equation in $C(\overline{\Omega})$.
Let $M > 0$ be fixed. Given $f \in C^\theta (\overline{\Omega})$, $\theta \in (0,1)$, let $v \in C^{2+\theta} (\overline{\Omega})$ be the solution of
\begin{align} \label{pN}
\begin{cases}
(-\Delta + M)v = f(x) & \mbox{ in } \Omega, \\
\dfrac{\partial v}{\partial \mathbf{n}} = 0 & \mbox{ on } \partial \Omega.
\end{cases}
\end{align}
We introduce the resolvent
$\mathcal{K} : C^\theta (\overline{\Omega}) \to
C^{2+\theta}_N (\overline{\Omega}) := \{ v \in C^{2+\theta}(\overline{\Omega}) : \frac{\partial v}{\partial \mathbf{n}} = 0 \mbox{ on } \partial \Omega \}$ for \eqref{pN}, i.e. $\mathcal{K}f=v$. It is well known (cf. \cite{GT01}) that $\mathcal{K}$ is bijective and homeomorphic. It is also well known (cf. \cite{Am76}) that $\mathcal{K}$ can be extended to a compact linear mapping from $C(\overline{\Omega})$ into $C^{1}(\overline{\Omega})$. In this way, as far as non-negative solutions are concerned, $(Q_{\mu, \epsilon})$ is reduced to
\begin{align*}
\mathcal{F}(\mu, v) :=
v - \mathcal{K}\left( Mv + \mu (b \epsilon^{q-2} v + g(x,v)) \right) = 0 \quad \mbox{ in }
C(\overline{\Omega}),
\end{align*}
where $g(x,s) = b(x)\{ (s+\epsilon)^{q-2}s - \epsilon^{q-2}s \} + a(x)|s|^{p-2}s$. We note that $g(x, \cdot)$ is $C^1$ at $s=0$, and $\frac{\partial g}{\partial s}(x,0)=0$, so that $\mathcal{F}$ has Fr\'echet derivatives $\mathcal{F}_v(\mu, 0)$ and $\mathcal{F}_{\mu v}(\mu, 0)$  given, respectively, by
\begin{align*}
& \mathcal{F}_v(\mu, 0) \varphi = \varphi - \mathcal{K}(M \varphi + \mu \epsilon^{q-2}b \varphi), \\
& \mathcal{F}_{\mu v}(\mu, 0) \varphi
= - \mathcal{K}(\epsilon^{q-2}b \varphi).
\end{align*}
We consider the eigenvalue problem
\begin{align} \label{160729epr}
\begin{cases}
-\Delta \varphi = \mu \epsilon^{q-2} b(x) \varphi & \mbox{in $\Omega$}, \\
\dfrac{\partial \varphi}{\partial \mathbf{n}} = 0 & \mbox{on $\partial
\Omega$},
\end{cases}
\end{align}
where $\epsilon$ is fixed, and $\mu$ is the eigenvalue parameter.
Since $b$ changes sign and $\int_\Omega b < 0$, this problem has exactly two principal eigenvalues, $\mu = 0$, and $\mu=\mu_\epsilon>0$ (cf. \cite{BL80}), which are both simple and possess positive eigenfunctions 	
$\varphi = 1$ and $\varphi=\varphi_\epsilon$, respectively, both satisfying $\|\varphi\|_\infty=1$.
Hence, the principal eigenvalues $\mu = 0$, and $\mu=\mu_\epsilon$ satisfy
\begin{align*}
\mathcal{N}(\mathcal{F}_v(0, 0)) = \langle 1 \rangle, \quad \mathcal{N}(\mathcal{F}_v(\mu_\epsilon, 0)) = \langle \varphi_\epsilon \rangle,
\end{align*}
and
\begin{align*}
\mathcal{F}_{\mu v}(0, 0) 1 \not\in \mathcal{R}(\mathcal{F}_v(0,0)), \quad
\mathcal{F}_{\mu v}(\mu_\epsilon, 0) \varphi_\epsilon \not\in \mathcal{R}(\mathcal{F}_v(\mu_\epsilon, 0)),
\end{align*}
where $\mathcal{N}(\cdot)$ and $\mathcal{R}(\cdot)$ denote the kernel and range of mappings, respectively. Indeed, the latter two assertions are verified using the conditions  that $\int_\Omega b < 0$ and $\int_\Omega b \varphi_\epsilon^2 > 0$, respectively. The local bifurcation theory 	
\cite[Theorem 1.7]{CR71} can now be applied.  Moreover, by the unilateral global bifurcation theory (\cite[Theorem 1.1]{Um10}, see also Lopez-Gomez \cite[Theorem 6.4.3]{LG01}), we infer that 	
$(Q_{\mu, \epsilon})$ has a component $\mathcal{C}_\epsilon = \{ (\mu, u)\}$ of non-negative solutions bifurcating at $(\mu_\epsilon, 0)$ in $\mathbb{R} \times C(\overline{\Omega})$, and there are no positive solutions of $(Q_{\mu, \epsilon})$ bifurcating at $(0,0)$, except $\Gamma_{00}$.
Moreover, $\mathcal{C}_\epsilon$ has the following properties:

\begin{lemma} \label{lem:Cep}
Let $\epsilon \in (0,1]$. Under the conditions in Theorem \ref{tp}, we have the following:
\begin{enumerate}
\item[(i)] There exists $\Lambda > 0$ such that $(Q_{\mu, \epsilon})$ has no positive solution for $\mu > \frac{\Lambda}{2}$.
Here, $\Lambda$ does not depend on $\epsilon \in (0,1]$.
%
%
\item[(ii)]
\begin{enumerate}
\item Assume that $\int_\Omega a > 0$. Then either $\mathcal{C}_\epsilon$ meets $\Gamma_{00}$ at $(0, c^*_\epsilon)$ or it does not meet $\Gamma_{00}$. In the later case,  $\mathcal{C}_\epsilon$ bifurcates from infinity. Moreover, $\mathcal{C}_\epsilon$ does not meet any $(0, c)$, except if $c=c^*_\epsilon$. 
\item Assume that $\int_\Omega a = 0$. Then $\mathcal{C}_\epsilon$ does not meet $\Gamma_{00}$. Consequently, $\mathcal{C}_\epsilon$ bifurcates from infinity.
\end{enumerate}
\item[(iii)] $\mathcal{C}_\epsilon$ does not meet any $(\mu, 0)$, except if $\mu = 0$ or $\mu= \mu_\epsilon$.
In particular, $\mathcal{C}_\epsilon$ meets $(0,0)$ if and only if $\mathcal{C}_\epsilon$ meets $(0, c^*_\epsilon)$.  Consequently, $\mathcal{C}_\epsilon \setminus \{ (0,0), (\mu_\epsilon, 0)\}$ is composed by positive solutions of $(Q_{\mu, \epsilon})$.
\end{enumerate}
Possible bifurcation diagrams for $(Q_{\mu, \epsilon})$ are depicted  in Figures \ref{fig16_0730ab} and \ref{fig16_0730aa} for the cases $\int_\Omega a > 0$ and
$\int_\Omega a = 0$, respectively.
\end{lemma}

\begin{proof}\strut
\begin{enumerate} 		
\item[(i)] The proof is similar to the one of \cite[Proposition 6.1]{RQU}, so we provide an outline of it. By $(H_0)$, we can choose an open ball $B$ centered at $x_0$ such that $\overline{B} \subset \Omega^{a}_{+} \cap \Omega^{b}_{+}$. We consider the Dirichlet eigenvalue problem
\begin{align} \label{eprDB}
\begin{cases}
-\Delta w = \mu a(x) w & \mbox{ in } B, \\
w = 0 & \mbox{ on } \partial B.
\end{cases}
\end{align}
Let $w_D \in C^2(\overline{B})$ be a positive eigenfunction associated with the first eigenvalue $\mu_D > 0$ of \eqref{eprDB}.  We extend $w_D$ to the  $\Omega$ by setting $w_D = 0$ in $\Omega\setminus \overline{B}$. Then, $w_D \in H^1(\Omega)$.

Let $v$ be a positive solution of $(Q_{\mu, \epsilon})$. By the divergence theorem, we deduce that $\int_B \nabla \cdot \left( v\nabla w_D \right) = \int_{\partial B} v \frac{\partial w_D}{\partial \mathbf{n}} < 0$. It follows that
\begin{align*}
\int_B \nabla v \nabla w_D - \mu_D \int_B a v w_D < 0.
\end{align*}
On the other hand,  by the definition of $v$, we see that
\begin{align*}
\int_B \nabla v \nabla w_D = \mu \int_B a v^{p-1} w_D + \mu \int_B b (v+\epsilon)^{q-2}v w_D.
\end{align*}
If $\mu \geq 1$, then we deduce that
\begin{align*}
0 & > \int_B v^{q-1} w_D  \left\{ \mu a v^{p-q} + \mu b \left( \frac{v}{v+\epsilon} \right)^{2-q} - \mu_D av^{2-q} \right\} \\
& \geq \int_B v^{q-1} w_D  \left\{ a v^{p-q} + \mu b \left( \frac{v}{v+\epsilon} \right)^{2-q} - \mu_D av^{2-q} \right\}.
\end{align*}
The rest of the proof follows as in \cite[Proposition 6.1]{RQU}.
Indeed, we can show that there exists $\overline{\mu} > 0$ such that if $\mu \geq \overline{\mu}$, $\epsilon \in (0, 1]$, $x\in B$ and $s\geq 0$, then
\begin{align*}
a(x) s^{p-q} + \mu b(x) \left( \frac{s}{s + \epsilon} \right)^{2-q} - \mu_D a(x) s^{2-q} \geq 0.
\end{align*}
Consequently, $\mu$ is bounded from above, uniformly in $\epsilon \in (0,1]$.
\item[(ii)] If $\int_\Omega a > 0$, then, thanks to the previous item,
we infer by Lemma \ref{lem:bf00} (i) that $\mathcal{C}_\epsilon$ bifurcates
from infinity if it does not meet $(0, c_\epsilon^*)$,
so assertion (ii)(a) follows.
Similarly, if $\int_\Omega a = 0$, then the previous item and Lemma \ref{lem:bf00} (ii) yield that $\mathcal{C}_\epsilon$ bifurcates from infinity,
so assertion (ii)(b) follows.

%
\item[(iii)] This
assertion is deduced from the fact that \eqref{160729epr} has exactly two principal eigenvalues $\mu = 0$ and $\mu=\mu_\epsilon$.
\end{enumerate}
\end{proof}

	\begin{figure}[!htb]
      \begin{center}
      \subfigure[$\mathcal{C}_\epsilon$ meets $\Gamma_{00}$ at $(0, c^*_\epsilon)$,
and $\mathcal{C}_\epsilon \setminus (\Gamma_0 \cup \Gamma_{00})$ is bounded.]{
{\unitlength 0.1in
\begin{picture}( 18.2500, 17.3700)( 10.4000,-21.3000)
%
\special{pn 20}%
\special{pa 2076 2130}%
\special{pa 2076 598}%
\special{fp}%
\special{sh 1}%
\special{pa 2076 598}%
\special{pa 2056 664}%
\special{pa 2076 650}%
\special{pa 2096 664}%
\special{pa 2076 598}%
\special{fp}%
%
\special{pn 8}%
\special{pa 1898 1900}%
\special{pa 2866 1900}%
\special{fp}%
\special{sh 1}%
\special{pa 2866 1900}%
\special{pa 2798 1880}%
\special{pa 2812 1900}%
\special{pa 2798 1920}%
\special{pa 2866 1900}%
\special{fp}%
\put(28.9200,-17.7100){\makebox(0,0){$\mu$}}%
\put(19.5500,-20.1200){\makebox(0,0){O}}%
\put(20.7500,-4.7300){\makebox(0,0){$\| v \|_{C(\overline{\Omega})}$}}%
%
\special{pn 8}%
\special{pa 2676 1900}%
\special{pa 2676 620}%
\special{dt 0.045}%
\put(26.7500,-20.1800){\makebox(0,0){$\Lambda$}}%
\put(23.3500,-19.9500){\makebox(0,0){$\mu_\epsilon$}}%
%
\special{pn 20}%
\special{pa 2380 1900}%
\special{pa 2384 1868}%
\special{pa 2388 1838}%
\special{pa 2394 1806}%
\special{pa 2398 1774}%
\special{pa 2404 1742}%
\special{pa 2408 1712}%
\special{pa 2420 1648}%
\special{pa 2434 1584}%
\special{pa 2442 1554}%
\special{pa 2450 1522}%
\special{pa 2456 1490}%
\special{pa 2464 1426}%
\special{pa 2466 1394}%
\special{pa 2466 1362}%
\special{pa 2462 1330}%
\special{pa 2456 1300}%
\special{pa 2446 1268}%
\special{pa 2434 1236}%
\special{pa 2418 1206}%
\special{pa 2400 1178}%
\special{pa 2380 1150}%
\special{pa 2360 1126}%
\special{pa 2336 1102}%
\special{pa 2310 1082}%
\special{pa 2284 1064}%
\special{pa 2256 1048}%
\special{pa 2226 1036}%
\special{pa 2196 1026}%
\special{pa 2164 1016}%
\special{pa 2134 1010}%
\special{pa 2070 998}%
\special{fp}%
\put(19.3800,-10.0400){\makebox(0,0){$c^*_\epsilon$}}%
\put(19.3800,-14.3400){\makebox(0,0){$\Gamma_{00}$}}%
\put(24.7600,-10.5700){\makebox(0,0){$\mathcal{C}_\epsilon$}}%
\put(25.4000,-18.1000){\makebox(0,0){$\Gamma_0$}}%
\end{picture}}%
          \label{fig16_0730b}
      }
      \hfill
      \subfigure[$\mathcal{C}_\epsilon$ meets $\Gamma_{00}$ at $c=c^*_\epsilon$, and  bifurcation from infinity occurs.]{
{\unitlength 0.1in
\begin{picture}( 17.8300, 17.5500)(  9.6100,-26.1000)
%
\special{pn 20}%
\special{pa 1996 2610}%
\special{pa 1996 1062}%
\special{fp}%
\special{sh 1}%
\special{pa 1996 1062}%
\special{pa 1976 1128}%
\special{pa 1996 1114}%
\special{pa 2016 1128}%
\special{pa 1996 1062}%
\special{fp}%
%
\special{pn 8}%
\special{pa 1830 2380}%
\special{pa 2744 2380}%
\special{fp}%
\special{sh 1}%
\special{pa 2744 2380}%
\special{pa 2678 2360}%
\special{pa 2692 2380}%
\special{pa 2678 2400}%
\special{pa 2744 2380}%
\special{fp}%
\put(27.6900,-22.4600){\makebox(0,0){$\mu$}}%
\put(18.8400,-24.9200){\makebox(0,0){O}}%
\put(19.9600,-9.3500){\makebox(0,0){$\| v \|_{C(\overline{\Omega})}$}}%
%
\special{pn 8}%
\special{pa 2564 2380}%
\special{pa 2564 1086}%
\special{dt 0.045}%
\put(25.6400,-24.9600){\makebox(0,0){$\Lambda$}}%
\put(22.4200,-24.7200){\makebox(0,0){$\mu_\epsilon$}}%
%
\special{pn 20}%
\special{pa 2284 2380}%
\special{pa 2292 2316}%
\special{pa 2298 2284}%
\special{pa 2302 2252}%
\special{pa 2306 2222}%
\special{pa 2312 2190}%
\special{pa 2316 2158}%
\special{pa 2328 2094}%
\special{pa 2336 2064}%
\special{pa 2342 2032}%
\special{pa 2350 2000}%
\special{pa 2356 1968}%
\special{pa 2360 1936}%
\special{pa 2364 1906}%
\special{pa 2366 1874}%
\special{pa 2366 1842}%
\special{pa 2364 1810}%
\special{pa 2358 1776}%
\special{pa 2350 1744}%
\special{pa 2338 1714}%
\special{pa 2324 1682}%
\special{pa 2308 1654}%
\special{pa 2290 1626}%
\special{pa 2270 1600}%
\special{pa 2248 1574}%
\special{pa 2196 1534}%
\special{pa 2170 1518}%
\special{pa 2110 1494}%
\special{pa 2080 1486}%
\special{pa 2048 1478}%
\special{pa 2016 1472}%
\special{pa 1992 1468}%
\special{fp}%
\put(18.6600,-14.7300){\makebox(0,0){$c^*_\epsilon$}}%
\put(18.6600,-19.0700){\makebox(0,0){$\Gamma_{00}$}}%
\put(24.5300,-16.6300){\makebox(0,0){$\mathcal{C}_\epsilon$}}%
\put(24.5600,-22.9000){\makebox(0,0){$\Gamma_0$}}%
%
\special{pn 20}%
\special{pa 2322 1670}%
\special{pa 2338 1640}%
\special{pa 2352 1610}%
\special{pa 2366 1578}%
\special{pa 2374 1548}%
\special{pa 2380 1518}%
\special{pa 2380 1490}%
\special{pa 2374 1460}%
\special{pa 2364 1432}%
\special{pa 2348 1402}%
\special{pa 2332 1374}%
\special{pa 2296 1318}%
\special{pa 2280 1288}%
\special{pa 2266 1258}%
\special{pa 2256 1228}%
\special{pa 2246 1196}%
\special{pa 2240 1166}%
\special{pa 2232 1102}%
\special{pa 2230 1070}%
\special{pa 2230 1062}%
\special{fp}%
\end{picture}}%
       \label{fig17_0309b}
      }
      \hfill
      \subfigure[$\mathcal{C}_\epsilon$ does not meet $\Gamma_{00}$, but bifurcation from infinity occurs.]{
{\unitlength 0.1in
\begin{picture}( 17.3000, 16.7400)( 14.4800,-21.9000)
%
\special{pn 8}%
\special{pa 2484 2190}%
\special{pa 2484 716}%
\special{fp}%
\special{sh 1}%
\special{pa 2484 716}%
\special{pa 2464 782}%
\special{pa 2484 768}%
\special{pa 2504 782}%
\special{pa 2484 716}%
\special{fp}%
%
\special{pn 8}%
\special{pa 2328 1970}%
\special{pa 3178 1970}%
\special{fp}%
\special{sh 1}%
\special{pa 3178 1970}%
\special{pa 3112 1950}%
\special{pa 3126 1970}%
\special{pa 3112 1990}%
\special{pa 3178 1970}%
\special{fp}%
\put(32.0100,-18.4400){\makebox(0,0){$\mu$}}%
\put(23.7900,-20.7600){\makebox(0,0){O}}%
\put(24.8300,-5.9600){\makebox(0,0){$\| v \|_{C(\overline{\Omega})}$}}%
%
\special{pn 20}%
\special{pa 2724 1970}%
\special{pa 2748 1912}%
\special{pa 2772 1852}%
\special{pa 2782 1822}%
\special{pa 2794 1792}%
\special{pa 2804 1760}%
\special{pa 2820 1696}%
\special{pa 2824 1664}%
\special{pa 2824 1632}%
\special{pa 2816 1602}%
\special{pa 2804 1576}%
\special{pa 2786 1550}%
\special{pa 2746 1498}%
\special{pa 2726 1470}%
\special{pa 2708 1442}%
\special{pa 2692 1412}%
\special{pa 2682 1380}%
\special{pa 2674 1348}%
\special{pa 2672 1316}%
\special{pa 2672 1286}%
\special{pa 2678 1254}%
\special{pa 2688 1224}%
\special{pa 2700 1194}%
\special{pa 2714 1164}%
\special{pa 2730 1134}%
\special{pa 2746 1106}%
\special{pa 2760 1076}%
\special{pa 2774 1048}%
\special{pa 2784 1018}%
\special{pa 2790 988}%
\special{pa 2792 958}%
\special{pa 2792 928}%
\special{pa 2788 896}%
\special{pa 2780 866}%
\special{pa 2760 802}%
\special{pa 2734 738}%
\special{pa 2728 726}%
\special{fp}%
%
\special{pn 8}%
\special{pa 3012 1970}%
\special{pa 3012 738}%
\special{dt 0.045}%
\put(30.1100,-20.8300){\makebox(0,0){$\Lambda$}}%
\put(27.1300,-20.6000){\makebox(0,0){$\mu_\epsilon$}}%
\put(23.4000,-13.5100){\makebox(0,0){$\Gamma_{00}$}}%
\put(28.1200,-13.4500){\makebox(0,0){$\mathcal{C}_\epsilon$}}%
\put(29.2000,-18.9400){\makebox(0,0){$\Gamma_0$}}%
\end{picture}}%
        \label{fig16_0730a}
      }
      \end{center}
      \caption{Possible bifurcation diagrams for $\mathcal{C}_\epsilon$: the case $\int_\Omega a > 0$.}
      \label{fig16_0730ab}
    \end{figure}
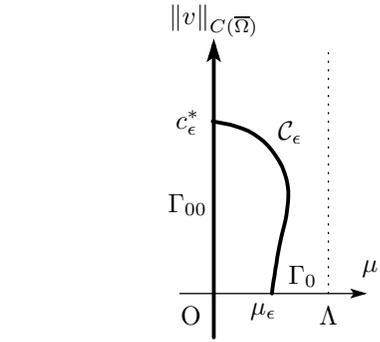
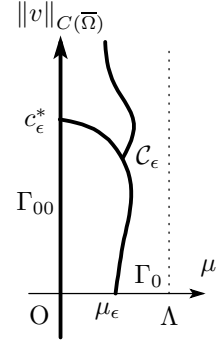
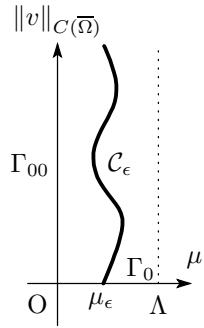


	 \begin{figure}[!htb]
  	   \begin{center}
{\unitlength 0.1in
\begin{picture}( 17.3000, 16.7400)( 14.4800,-21.9000)
%
\special{pn 8}%
\special{pa 2484 2190}%
\special{pa 2484 716}%
\special{fp}%
\special{sh 1}%
\special{pa 2484 716}%
\special{pa 2464 782}%
\special{pa 2484 768}%
\special{pa 2504 782}%
\special{pa 2484 716}%
\special{fp}%
%
\special{pn 8}%
\special{pa 2328 1970}%
\special{pa 3178 1970}%
\special{fp}%
\special{sh 1}%
\special{pa 3178 1970}%
\special{pa 3112 1950}%
\special{pa 3126 1970}%
\special{pa 3112 1990}%
\special{pa 3178 1970}%
\special{fp}%
\put(32.0100,-18.4400){\makebox(0,0){$\mu$}}%
\put(23.7900,-20.7600){\makebox(0,0){O}}%
\put(24.8300,-5.9600){\makebox(0,0){$\| v \|_{C(\overline{\Omega})}$}}%
%
\special{pn 20}%
\special{pa 2724 1970}%
\special{pa 2748 1912}%
\special{pa 2772 1852}%
\special{pa 2782 1822}%
\special{pa 2794 1792}%
\special{pa 2804 1760}%
\special{pa 2820 1696}%
\special{pa 2824 1664}%
\special{pa 2824 1632}%
\special{pa 2816 1602}%
\special{pa 2804 1576}%
\special{pa 2786 1550}%
\special{pa 2746 1498}%
\special{pa 2726 1470}%
\special{pa 2708 1442}%
\special{pa 2692 1412}%
\special{pa 2682 1380}%
\special{pa 2674 1348}%
\special{pa 2672 1316}%
\special{pa 2672 1286}%
\special{pa 2678 1254}%
\special{pa 2688 1224}%
\special{pa 2700 1194}%
\special{pa 2714 1164}%
\special{pa 2730 1134}%
\special{pa 2746 1106}%
\special{pa 2760 1076}%
\special{pa 2774 1048}%
\special{pa 2784 1018}%
\special{pa 2790 988}%
\special{pa 2792 958}%
\special{pa 2792 928}%
\special{pa 2788 896}%
\special{pa 2780 866}%
\special{pa 2760 802}%
\special{pa 2734 738}%
\special{pa 2728 726}%
\special{fp}%
%
\special{pn 8}%
\special{pa 3012 1970}%
\special{pa 3012 738}%
\special{dt 0.045}%
\put(30.1100,-20.8300){\makebox(0,0){$\Lambda$}}%
\put(27.1300,-20.6000){\makebox(0,0){$\mu_\epsilon$}}%
\put(23.4000,-13.5100){\makebox(0,0){$\Gamma_{00}$}}%
\put(28.1200,-13.4500){\makebox(0,0){$\mathcal{C}_\epsilon$}}%
\put(29.2000,-18.9400){\makebox(0,0){$\Gamma_0$}}%
\end{picture}}%
	  \caption{Possible bifurcation diagram for $\mathcal{C}_\epsilon$: the case $\int_\Omega a = 0$.}
	\label{fig16_0730aa}
	  \end{center}
	    \end{figure}
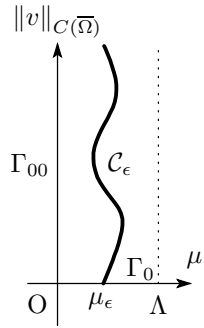

\section{Proof of Theorem \ref{tp}} \label{sec:proof}

In the sequel we study the limiting behavior of $\mathcal{C}_\epsilon$ as $\epsilon \to 0^+$. To this end, we recall some definitions. 	
Let $X$ be a complete metric space. Given $E_n \subset X$, $n\geq 1$, we set
\begin{align*}
& \liminf_{n\to \infty}E_n := \{ x \in X : \lim_{n\to \infty}{\rm dist}\, (x, E_n) = 0 \}, \\
& \limsup_{n\to \infty}E_n := \{ x \in X : \liminf_{n\to \infty}{\rm dist}\, (x, E_n) = 0 \},
\end{align*}
where ${\rm dist}\, (x, A)$ is the usual distance function for a set $A$. It is well known from Whyburn \cite[(9.12)Theorem]{Wh64} that if $\{ E_n \}$ is a sequence of connected sets in $X$ satisfying
\begin{align}
& \mbox{$\displaystyle{\liminf_{n\to \infty}}\ E_n \not= \emptyset$},  \label{linfnon} \\
& \mbox{$\displaystyle{\bigcup_{n\geq 1}} E_n$ is precompact in $X$}, \label{pcom}
\end{align}
then $\displaystyle{\limsup_{n\to \infty}}\ E_n$ is nonempty, closed, and connected.

Let $\rho > 0$, and set
\begin{align*}
\mathcal{C}_{\epsilon, \rho} := \overline{\mathcal{C}_\epsilon \cap ((0, \Lambda) \times B_\rho)},
\end{align*}
where
$B_\rho := \{ u \in C(\overline{\Omega}) : \| u \|_{C(\overline{\Omega})} \leq \rho \}$ is a closed ball in $C(\overline{\Omega})$, and $\Lambda$ is the positive constant provided by Lemma \ref{lem:Cep}(i).
Then, $\mathcal{C}_{\epsilon, \rho}$ is a bounded component satisfying
(see Figure \ref{fig17_0529abc}):
\begin{itemize}
\item $\mathcal{C}_{\epsilon, \rho}$ contains only $(\mu_\epsilon, 0)$ on $\Gamma_0$
(by Lemma \ref{lem:Cep}(iii));
\item $\mathcal{C}_{\epsilon, \rho}$ does not meet $\mu = \Lambda$ (by Lemma \ref{lem:Cep}(i)).
\item $\mathcal{C}_{\epsilon, \rho}$ does not meet $\Gamma_{00}$, except at $(0, c^*_\epsilon)$ (by Lemma \ref{lem:Cep}(ii));
\item $\mathcal{C}_{\epsilon, \rho}$ contains either $(0, c^*_\epsilon)$ or some $(\mu, v)$ such that $\mu \in (0, \Lambda)$ and $\| v \|_{C(\overline{\Omega})} = \rho$.
\end{itemize}
Letting $X = [0, \Lambda] \times B_\rho$,  $\epsilon_n \to 0^+$, and $E_n = \mathcal{C}_{\epsilon_n, \rho}$, we shall verify \eqref{linfnon} and \eqref{pcom}.

	\begin{figure}[!htb]
      \begin{center}
      \subfigure[$\mathcal{C}_{\epsilon, \rho}$ meets $\Gamma_{00}$ at $(0, c^*_\epsilon)$ but does not meet $\| v \|_{C(\overline{\Omega})} = \rho$.]{
{\unitlength 0.1in
\begin{picture}( 22.7500, 15.4100)( 13.2200,-18.9000)
%
\special{pn 8}%
\special{pa 2340 1890}%
\special{pa 2340 522}%
\special{fp}%
\special{sh 1}%
\special{pa 2340 522}%
\special{pa 2320 588}%
\special{pa 2340 574}%
\special{pa 2360 588}%
\special{pa 2340 522}%
\special{fp}%
%
\special{pn 8}%
\special{pa 2148 1700}%
\special{pa 3598 1700}%
\special{fp}%
\special{sh 1}%
\special{pa 3598 1700}%
\special{pa 3530 1680}%
\special{pa 3544 1700}%
\special{pa 3530 1720}%
\special{pa 3598 1700}%
\special{fp}%
%
\special{pn 8}%
\special{pn 8}%
\special{pa 2340 760}%
\special{pa 2348 760}%
\special{fp}%
\special{pa 2385 760}%
\special{pa 2393 760}%
\special{fp}%
\special{pa 2430 760}%
\special{pa 2438 760}%
\special{fp}%
\special{pa 2475 760}%
\special{pa 2483 760}%
\special{fp}%
\special{pa 2520 760}%
\special{pa 2528 760}%
\special{fp}%
\special{pa 2565 760}%
\special{pa 2573 760}%
\special{fp}%
\special{pa 2610 760}%
\special{pa 2618 760}%
\special{fp}%
\special{pa 2655 760}%
\special{pa 2663 760}%
\special{fp}%
\special{pa 2700 760}%
\special{pa 2708 760}%
\special{fp}%
\special{pa 2745 760}%
\special{pa 2753 760}%
\special{fp}%
\special{pa 2790 760}%
\special{pa 2798 760}%
\special{fp}%
\special{pa 2834 760}%
\special{pa 2842 760}%
\special{fp}%
\special{pa 2879 760}%
\special{pa 2887 760}%
\special{fp}%
\special{pa 2924 760}%
\special{pa 2932 760}%
\special{fp}%
\special{pa 2969 760}%
\special{pa 2977 760}%
\special{fp}%
\special{pa 3014 760}%
\special{pa 3022 760}%
\special{fp}%
\special{pa 3059 760}%
\special{pa 3067 760}%
\special{fp}%
\special{pa 3104 760}%
\special{pa 3112 760}%
\special{fp}%
\special{pa 3149 760}%
\special{pa 3157 760}%
\special{fp}%
\special{pa 3194 760}%
\special{pa 3202 760}%
\special{fp}%
\special{pa 3239 760}%
\special{pa 3247 760}%
\special{fp}%
\special{pa 3284 760}%
\special{pa 3292 760}%
\special{fp}%
\special{pa 3296 793}%
\special{pa 3296 801}%
\special{fp}%
\special{pa 3296 838}%
\special{pa 3296 846}%
\special{fp}%
\special{pa 3296 883}%
\special{pa 3296 891}%
\special{fp}%
\special{pa 3296 928}%
\special{pa 3296 936}%
\special{fp}%
\special{pa 3296 973}%
\special{pa 3296 981}%
\special{fp}%
\special{pa 3296 1018}%
\special{pa 3296 1026}%
\special{fp}%
\special{pa 3296 1063}%
\special{pa 3296 1071}%
\special{fp}%
\special{pa 3296 1108}%
\special{pa 3296 1116}%
\special{fp}%
\special{pa 3296 1153}%
\special{pa 3296 1161}%
\special{fp}%
\special{pa 3296 1198}%
\special{pa 3296 1206}%
\special{fp}%
\special{pa 3296 1242}%
\special{pa 3296 1250}%
\special{fp}%
\special{pa 3296 1287}%
\special{pa 3296 1295}%
\special{fp}%
\special{pa 3296 1332}%
\special{pa 3296 1340}%
\special{fp}%
\special{pa 3296 1377}%
\special{pa 3296 1385}%
\special{fp}%
\special{pa 3296 1422}%
\special{pa 3296 1430}%
\special{fp}%
\special{pa 3296 1467}%
\special{pa 3296 1475}%
\special{fp}%
\special{pa 3296 1512}%
\special{pa 3296 1520}%
\special{fp}%
\special{pa 3296 1557}%
\special{pa 3296 1565}%
\special{fp}%
\special{pa 3296 1602}%
\special{pa 3296 1610}%
\special{fp}%
\special{pa 3296 1647}%
\special{pa 3296 1655}%
\special{fp}%
\special{pa 3296 1692}%
\special{pa 3296 1700}%
\special{fp}%
\put(22.3600,-17.9200){\makebox(0,0){O}}%
\put(32.8900,-17.8600){\makebox(0,0){$\Lambda$}}%
\put(37.4400,-15.6600){\makebox(0,0){$\mu$}}%
\put(21.2700,-7.6000){\makebox(0,0){$\rho$}}%
\put(23.3700,-4.2900){\makebox(0,0){$\| v\|_{C(\overline{\Omega})}$}}%
%
\special{pn 20}%
\special{pa 2790 1700}%
\special{pa 2812 1640}%
\special{pa 2832 1580}%
\special{pa 2840 1548}%
\special{pa 2848 1518}%
\special{pa 2854 1486}%
\special{pa 2858 1454}%
\special{pa 2860 1420}%
\special{pa 2860 1388}%
\special{pa 2852 1320}%
\special{pa 2846 1288}%
\special{pa 2836 1254}%
\special{pa 2826 1222}%
\special{pa 2812 1192}%
\special{pa 2780 1136}%
\special{pa 2760 1112}%
\special{pa 2738 1090}%
\special{pa 2714 1072}%
\special{pa 2688 1056}%
\special{pa 2662 1042}%
\special{pa 2634 1032}%
\special{pa 2604 1022}%
\special{pa 2540 1010}%
\special{pa 2506 1006}%
\special{pa 2472 1004}%
\special{pa 2402 1000}%
\special{pa 2368 1000}%
\special{pa 2340 998}%
\special{fp}%
\put(21.2000,-10.2200){\makebox(0,0){$c^*_\epsilon$}}%
\put(30.2900,-10.9500){\makebox(0,0){$\mathcal{C}_{\epsilon, \rho}$}}%
\put(27.3600,-17.9900){\makebox(0,0){$\mu_\epsilon$}}%
\end{picture}}%
          \label{fig17_0529a}
      }
      \hfill
      \subfigure[$\mathcal{C}_{\epsilon, \rho}$ meets both $\Gamma_{00}$ at $c=c^*_\epsilon$ and $\| v \|_{C(\overline{\Omega})} = \rho$.]{
{\unitlength 0.1in
\begin{picture}( 21.7500, 15.4200)( 17.1500,-22.8000)
%
\special{pn 8}%
\special{pa 2724 2280}%
\special{pa 2724 932}%
\special{fp}%
\special{sh 1}%
\special{pa 2724 932}%
\special{pa 2704 998}%
\special{pa 2724 984}%
\special{pa 2744 998}%
\special{pa 2724 932}%
\special{fp}%
%
\special{pn 8}%
\special{pa 2548 2094}%
\special{pa 3890 2094}%
\special{fp}%
\special{sh 1}%
\special{pa 3890 2094}%
\special{pa 3824 2074}%
\special{pa 3838 2094}%
\special{pa 3824 2114}%
\special{pa 3890 2094}%
\special{fp}%
%
\special{pn 8}%
\special{pn 8}%
\special{pa 2724 1166}%
\special{pa 2732 1166}%
\special{fp}%
\special{pa 2769 1166}%
\special{pa 2777 1166}%
\special{fp}%
\special{pa 2814 1166}%
\special{pa 2822 1166}%
\special{fp}%
\special{pa 2860 1166}%
\special{pa 2868 1166}%
\special{fp}%
\special{pa 2905 1166}%
\special{pa 2913 1166}%
\special{fp}%
\special{pa 2950 1166}%
\special{pa 2958 1166}%
\special{fp}%
\special{pa 2995 1166}%
\special{pa 3003 1166}%
\special{fp}%
\special{pa 3040 1166}%
\special{pa 3048 1166}%
\special{fp}%
\special{pa 3086 1166}%
\special{pa 3094 1166}%
\special{fp}%
\special{pa 3131 1166}%
\special{pa 3139 1166}%
\special{fp}%
\special{pa 3176 1166}%
\special{pa 3184 1166}%
\special{fp}%
\special{pa 3221 1166}%
\special{pa 3229 1166}%
\special{fp}%
\special{pa 3266 1166}%
\special{pa 3274 1166}%
\special{fp}%
\special{pa 3312 1166}%
\special{pa 3320 1166}%
\special{fp}%
\special{pa 3357 1166}%
\special{pa 3365 1166}%
\special{fp}%
\special{pa 3402 1166}%
\special{pa 3410 1166}%
\special{fp}%
\special{pa 3447 1166}%
\special{pa 3455 1166}%
\special{fp}%
\special{pa 3492 1166}%
\special{pa 3500 1166}%
\special{fp}%
\special{pa 3538 1166}%
\special{pa 3546 1166}%
\special{fp}%
\special{pa 3583 1166}%
\special{pa 3591 1166}%
\special{fp}%
\special{pa 3612 1182}%
\special{pa 3612 1190}%
\special{fp}%
\special{pa 3612 1227}%
\special{pa 3612 1235}%
\special{fp}%
\special{pa 3612 1272}%
\special{pa 3612 1280}%
\special{fp}%
\special{pa 3612 1318}%
\special{pa 3612 1326}%
\special{fp}%
\special{pa 3612 1363}%
\special{pa 3612 1371}%
\special{fp}%
\special{pa 3612 1408}%
\special{pa 3612 1416}%
\special{fp}%
\special{pa 3612 1453}%
\special{pa 3612 1461}%
\special{fp}%
\special{pa 3612 1498}%
\special{pa 3612 1506}%
\special{fp}%
\special{pa 3612 1544}%
\special{pa 3612 1552}%
\special{fp}%
\special{pa 3612 1589}%
\special{pa 3612 1597}%
\special{fp}%
\special{pa 3612 1634}%
\special{pa 3612 1642}%
\special{fp}%
\special{pa 3612 1679}%
\special{pa 3612 1687}%
\special{fp}%
\special{pa 3612 1724}%
\special{pa 3612 1732}%
\special{fp}%
\special{pa 3612 1770}%
\special{pa 3612 1778}%
\special{fp}%
\special{pa 3612 1815}%
\special{pa 3612 1823}%
\special{fp}%
\special{pa 3612 1860}%
\special{pa 3612 1868}%
\special{fp}%
\special{pa 3612 1905}%
\special{pa 3612 1913}%
\special{fp}%
\special{pa 3612 1950}%
\special{pa 3612 1958}%
\special{fp}%
\special{pa 3612 1996}%
\special{pa 3612 2004}%
\special{fp}%
\special{pa 3612 2041}%
\special{pa 3612 2049}%
\special{fp}%
\special{pa 3612 2086}%
\special{pa 3612 2094}%
\special{fp}%
\put(26.2900,-21.8300){\makebox(0,0){O}}%
\put(36.0500,-21.7700){\makebox(0,0){$\Lambda$}}%
\put(40.2600,-19.6100){\makebox(0,0){$\mu$}}%
\put(25.2700,-11.6600){\makebox(0,0){$\rho$}}%
\put(27.3000,-8.1800){\makebox(0,0){$\| v\|_{C(\overline{\Omega})}$}}%
%
\special{pn 20}%
\special{pa 3142 2094}%
\special{pa 3152 2064}%
\special{pa 3162 2032}%
\special{pa 3172 2002}%
\special{pa 3188 1940}%
\special{pa 3196 1910}%
\special{pa 3204 1846}%
\special{pa 3206 1812}%
\special{pa 3206 1780}%
\special{pa 3204 1746}%
\special{pa 3200 1712}%
\special{pa 3192 1678}%
\special{pa 3184 1646}%
\special{pa 3172 1614}%
\special{pa 3160 1584}%
\special{pa 3144 1556}%
\special{pa 3126 1530}%
\special{pa 3106 1506}%
\special{pa 3084 1484}%
\special{pa 3060 1468}%
\special{pa 3034 1452}%
\special{pa 3006 1440}%
\special{pa 2978 1430}%
\special{pa 2946 1422}%
\special{pa 2914 1416}%
\special{pa 2882 1412}%
\special{pa 2848 1408}%
\special{pa 2812 1406}%
\special{pa 2778 1404}%
\special{pa 2742 1402}%
\special{pa 2724 1402}%
\special{fp}%
\put(25.2100,-14.2500){\makebox(0,0){$c^*_\epsilon$}}%
\put(34.0700,-16.4000){\makebox(0,0){$\mathcal{C}_{\epsilon, \rho}$}}%
\put(30.9200,-21.9000){\makebox(0,0){$\mu_\epsilon$}}%
%
\special{pn 20}%
\special{pa 3176 1626}%
\special{pa 3200 1598}%
\special{pa 3222 1572}%
\special{pa 3242 1546}%
\special{pa 3262 1518}%
\special{pa 3280 1490}%
\special{pa 3296 1462}%
\special{pa 3308 1434}%
\special{pa 3318 1406}%
\special{pa 3322 1376}%
\special{pa 3324 1346}%
\special{pa 3322 1316}%
\special{pa 3316 1286}%
\special{pa 3308 1254}%
\special{pa 3288 1190}%
\special{pa 3276 1160}%
\special{fp}%
\put(32.8400,-10.2200){\makebox(0,0){$(\mu, v)$}}%
\end{picture}}%
       \label{fig17_0529b}
      }
      \hfill
      \subfigure[$\mathcal{C}_{\epsilon, \rho}$ does not meet $\Gamma_{00}$ but does meet $\| v \|_{C(\overline{\Omega})} = \rho$.]{
{\unitlength 0.1in
\begin{picture}( 21.8900, 15.2300)( 17.7000,-22.1000)
%
\special{pn 8}%
\special{pa 2786 2210}%
\special{pa 2786 880}%
\special{fp}%
\special{sh 1}%
\special{pa 2786 880}%
\special{pa 2766 946}%
\special{pa 2786 932}%
\special{pa 2806 946}%
\special{pa 2786 880}%
\special{fp}%
%
\special{pn 8}%
\special{pa 2608 2026}%
\special{pa 3960 2026}%
\special{fp}%
\special{sh 1}%
\special{pa 3960 2026}%
\special{pa 3892 2006}%
\special{pa 3906 2026}%
\special{pa 3892 2046}%
\special{pa 3960 2026}%
\special{fp}%
%
\special{pn 8}%
\special{pn 8}%
\special{pa 2786 1110}%
\special{pa 2794 1110}%
\special{fp}%
\special{pa 2831 1110}%
\special{pa 2839 1110}%
\special{fp}%
\special{pa 2876 1110}%
\special{pa 2884 1110}%
\special{fp}%
\special{pa 2921 1110}%
\special{pa 2929 1110}%
\special{fp}%
\special{pa 2966 1110}%
\special{pa 2974 1110}%
\special{fp}%
\special{pa 3011 1110}%
\special{pa 3019 1110}%
\special{fp}%
\special{pa 3056 1110}%
\special{pa 3064 1110}%
\special{fp}%
\special{pa 3101 1110}%
\special{pa 3109 1110}%
\special{fp}%
\special{pa 3146 1110}%
\special{pa 3154 1110}%
\special{fp}%
\special{pa 3191 1110}%
\special{pa 3199 1110}%
\special{fp}%
\special{pa 3236 1110}%
\special{pa 3244 1110}%
\special{fp}%
\special{pa 3281 1110}%
\special{pa 3289 1110}%
\special{fp}%
\special{pa 3326 1110}%
\special{pa 3334 1110}%
\special{fp}%
\special{pa 3371 1110}%
\special{pa 3379 1110}%
\special{fp}%
\special{pa 3416 1110}%
\special{pa 3424 1110}%
\special{fp}%
\special{pa 3461 1110}%
\special{pa 3469 1110}%
\special{fp}%
\special{pa 3506 1110}%
\special{pa 3514 1110}%
\special{fp}%
\special{pa 3551 1110}%
\special{pa 3559 1110}%
\special{fp}%
\special{pa 3596 1110}%
\special{pa 3604 1110}%
\special{fp}%
\special{pa 3641 1110}%
\special{pa 3649 1110}%
\special{fp}%
\special{pa 3678 1118}%
\special{pa 3678 1126}%
\special{fp}%
\special{pa 3678 1163}%
\special{pa 3678 1171}%
\special{fp}%
\special{pa 3678 1208}%
\special{pa 3678 1216}%
\special{fp}%
\special{pa 3678 1253}%
\special{pa 3678 1261}%
\special{fp}%
\special{pa 3678 1298}%
\special{pa 3678 1306}%
\special{fp}%
\special{pa 3678 1343}%
\special{pa 3678 1351}%
\special{fp}%
\special{pa 3678 1388}%
\special{pa 3678 1396}%
\special{fp}%
\special{pa 3678 1433}%
\special{pa 3678 1441}%
\special{fp}%
\special{pa 3678 1478}%
\special{pa 3678 1486}%
\special{fp}%
\special{pa 3678 1523}%
\special{pa 3678 1531}%
\special{fp}%
\special{pa 3678 1568}%
\special{pa 3678 1576}%
\special{fp}%
\special{pa 3678 1613}%
\special{pa 3678 1621}%
\special{fp}%
\special{pa 3678 1658}%
\special{pa 3678 1666}%
\special{fp}%
\special{pa 3678 1703}%
\special{pa 3678 1711}%
\special{fp}%
\special{pa 3678 1748}%
\special{pa 3678 1756}%
\special{fp}%
\special{pa 3678 1793}%
\special{pa 3678 1801}%
\special{fp}%
\special{pa 3678 1838}%
\special{pa 3678 1846}%
\special{fp}%
\special{pa 3678 1883}%
\special{pa 3678 1891}%
\special{fp}%
\special{pa 3678 1928}%
\special{pa 3678 1936}%
\special{fp}%
\special{pa 3678 1973}%
\special{pa 3678 1981}%
\special{fp}%
\special{pa 3678 2018}%
\special{pa 3678 2026}%
\special{fp}%
\put(26.9000,-21.1500){\makebox(0,0){O}}%
\put(36.7200,-21.0900){\makebox(0,0){$\Lambda$}}%
\put(40.9500,-18.9500){\makebox(0,0){$\mu$}}%
\put(25.8800,-11.1000){\makebox(0,0){$\rho$}}%
\put(27.8500,-7.6700){\makebox(0,0){$\| v\|_{C(\overline{\Omega})}$}}%
\put(34.7300,-15.7800){\makebox(0,0){$\mathcal{C}_{\epsilon, \rho}$}}%
\put(32.5600,-21.1700){\makebox(0,0){$\mu_\epsilon$}}%
\put(33.4900,-9.6800){\makebox(0,0){$(\mu, v)$}}%
%
\special{pn 20}%
\special{pa 3226 2024}%
\special{pa 3250 1998}%
\special{pa 3272 1972}%
\special{pa 3294 1944}%
\special{pa 3314 1918}%
\special{pa 3330 1892}%
\special{pa 3342 1864}%
\special{pa 3350 1836}%
\special{pa 3352 1808}%
\special{pa 3350 1780}%
\special{pa 3346 1752}%
\special{pa 3336 1722}%
\special{pa 3324 1692}%
\special{pa 3310 1664}%
\special{pa 3294 1634}%
\special{pa 3278 1602}%
\special{pa 3242 1542}%
\special{pa 3226 1510}%
\special{pa 3212 1480}%
\special{pa 3198 1448}%
\special{pa 3188 1418}%
\special{pa 3180 1386}%
\special{pa 3178 1356}%
\special{pa 3178 1324}%
\special{pa 3180 1294}%
\special{pa 3186 1262}%
\special{pa 3194 1230}%
\special{pa 3204 1200}%
\special{pa 3228 1138}%
\special{pa 3242 1106}%
\special{fp}%
\end{picture}}%
        \label{fig17_0529c}
      }
      \end{center}
      \caption{Three possibilities for the bounded component $\mathcal{C}_{\epsilon, \rho}$.}
      \label{fig17_0529abc}
    \end{figure}
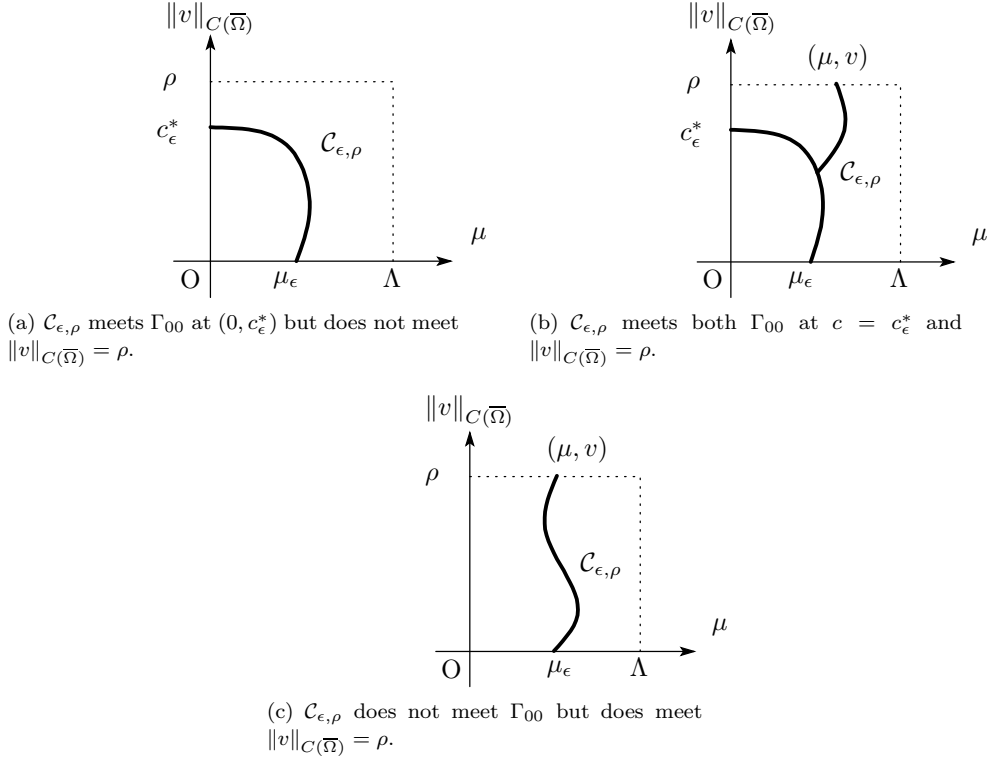

We note from \eqref{160729epr} that $\mu_1 = \mu_{\epsilon_n} \epsilon_n^{q-2}$, so that $\mu_{\epsilon_n} \to 0$. It follows that $(0,0) \in \liminf_{n\to \infty} \mathcal{C}_{\epsilon_n, \rho}$, since $(\mu_{\epsilon_n}, 0) \in \mathcal{C}_{\epsilon_n, \rho}$. In particular, we obtain assertion \eqref{linfnon}.

The boundedness of $\mathcal{C}_{\epsilon_n, \rho}$ implies that $\bigcup_{n}  \mathcal{C}_{\epsilon_n, \rho}$ is precompact. Indeed, for any $\{ (\mu_k, v_k) \} \subset \bigcup_{n} \mathcal{C}_{\varepsilon_n, \rho}$ the sequence $\epsilon_n$ has a subsequence $\epsilon_{n_k}$ such that $(\mu_k, v_k) \in \mathcal{C}_{\varepsilon_{n_k}}$, where $\epsilon_{n_k} \in (0,1]$. Then, by elliptic regularity, we deduce that $v_k \in C^2(\overline{\Omega})$,
$\Vert v_k \Vert_{C^{1}(\overline{\Omega})}$ is bounded, and
\begin{align} \label{pr:vk}
\begin{cases}
-\Delta v_k = \mu_k \left(
b(x) (v + \epsilon_{n_k})^{q-2}v_k + a(x) v_k^{p-1} \right) & \mbox{in $\Omega$}, \\
\dfrac{\partial v_k}{\partial \mathbf{n}} = 0 & \mbox{on $\partial \Omega$}.
\end{cases}
\end{align}
Since $(\mu_k, v_k)$ and $\epsilon_{n_k}$ are bounded, using the compact embedding $C^{1}(\overline{\Omega})\subset C(\overline{\Omega})$ we deduce that
$\{ (\mu_k, v_k) \}$ has a convergent subsequence in $[0, \Lambda] \times B_\rho$, . Thus, assertion \eqref{pcom} is verified.
%

We may now apply Whyburn's result, so that \cite[(9.12)Theorem]{Wh64} implies that
\begin{align} \label{defC0r}
\mathcal{C}_{0, \rho} := \limsup_{n\to \infty} \mathcal{C}_{\epsilon_n, \rho}
\end{align}
is nonempty, closed and connected, i.e., it is a nonempty component in
$[0, \Lambda] \times B_\rho$.  Moreover, we shall show that $\mathcal{C}_{0, \rho}$ consists of non-negative solutions of $(Q_\mu)$, and
\begin{align} \label{includeO}
(0,0) \in \liminf_{n\to \infty} \mathcal{C}_{\epsilon_n, \rho} \subset \mathcal{C}_{0, \rho}.
\end{align}
The proof of 	
these facts is similar to the verification of the precompactness of $\bigcup_{n}  \mathcal{C}_{\epsilon_n, \rho}$. Indeed,
given $(\mu, v) \in \mathcal{C}_{0, \rho}$,
the sequence $\epsilon_n \to 0^+$ has a subsequence, still denoted by the same notation, such that there exist $(\mu_n, v_n) \in \mathcal{C}_{\varepsilon_n, \rho}$ satisfying
$(\mu_n, v_n) \to (\mu, v)$ in $\mathbb{R} \times C(\overline{\Omega})$.
It follows, by a bootstrap argument based on elliptic regularity, that $v_n \to v$ in $C^{1}(\overline{\Omega})$, so that $v$ is a non-negative weak solution of $(Q_\mu)$ (see \eqref{pr:vk}), and eventually, a non-negative solution in $C^{2+\theta}(\overline{\Omega})$ for some $\theta \in (0,1)$, by elliptic regularity.

Next, we shall 	
prove that $\mathcal{C}_{0, \rho}$ is nontrivial, i.e., we exclude the possibility that $\mathcal{C}_{0, \rho} \subset \Gamma_0 \cup \Gamma_{00}$. Let
$\rho, M$ be such that $0<M<c^*_0 < \rho$. Then,
we find from \eqref{cepto0} and \eqref{includeO}
that $\mathcal{C}_{0, \rho}$ joins $(0,0)$ to either $(0, c^*_0)$ or $(\mu, v) \in
[0, \Lambda]\times B_\rho$. Since $\mathcal{C}_{0, \rho}$ is connected, the intermediate value theorem shows the existence of $(\mu_0, v_0) \in \mathcal{C}_{0, \rho}$ such that $\| v_0 \|_{C(\overline{\Omega})} = M$. By definition, the sequence $\epsilon_n \to 0^+$ has a subsequence, still denoted by the same notation, such that
there exist $(\mu_n, v_n) \in \mathcal{C}_{\epsilon_n, \rho}$ with $0< \mu_n \to \mu_0$ and $v_n \to v_0$ in $C(\overline{\Omega})$. Assume by contradiction that $\mu_0 = 0$. Then, $v_0 = M$, so that $v_n \to M$ in $C(\overline{\Omega})$. However,
applying the divergence theorem to the solution $v_n$ of $(Q_{\mu_n, \epsilon_n})$, we obtain
\begin{align*} 
0 = \mu_n^{-1} \int_\Omega \left( -\Delta v_n \right) =
\int_\Omega \left\{ b(x)(v_n + \epsilon_n)^{q-2} v_n + a(x) v_n^{p-1} \right\},
\end{align*}
so that
passing to the limit, we deduce that
\begin{align} \label{Mfrml}
0=M^{q-1} \int_\Omega b + M^{p-1} \int_\Omega a,
\end{align}
and thus that $M=c^*_0$, which is impossible. Consequently, $\mu_0 > 0$, and thus,
$\mathcal{C}_{0, \rho}$ is nontrivial.

Using $(H_3)$ and \cite[Proposition 3.3]{RQU}, we infer  
that the nontrivial non-negative solutions set of $(P_\lambda)$ does not meet $\Gamma_0$ at any $\lambda \neq 0$, so that neither does the one of $(Q_\mu)$. Similarly to Lemma \ref{lem:Cep}(ii) and (iii), we see
from Lemma \ref{lem:bf00} (see Remark \ref{rem:lem4}) that the nontrivial non-negative solutions set of $(Q_\mu)$ does not meet $\Gamma_{00}$ at any $c \not= c^*_0$,  and \eqref{includeO} ensures that  $\mathcal{C}_{0, \rho}$ joins $(0,0)$ to either $(0, c^*_0)$ 
or some $(\mu_1, v_1)$ such that $\mu_1 \in (0, \Lambda)$ and $\| v_1 \|_{C(\overline{\Omega})} = \rho$. Since $\rho$ is arbitrary, we obtain a component $\mathcal{C}_0^\prime$ of non-negative solutions of $(Q_\mu)$ such that (see Figure \ref{fig16_0730cd}): 
\begin{enumerate}
\item[(c1)] $\mathcal{C}_0^\prime \setminus \{ (0,0), (0, c^*_0) \}$ consists of nontrivial non-negative solutions;
\item[(c2)] $\mathcal{C}_0^\prime$ joins $(0,0)$ to either $(0, c^*_0)$ or  $(0, \infty)$.
\item[(c3)] If
$(\mu, v) \in \mathcal{C}^\prime_0$, then $\mu \geq 0$.
\end{enumerate}
Note that the second possibility in (c2) follows from the following \textit{a priori} upper bound for positive solutions of $(Q_{\mu, \epsilon})$: given $\overline{\mu} \in (0,1)$, there exists $C_{\overline{\mu}} > 0$ such that $v \leq C_{\overline{\mu}}$ on $\overline{\Omega}$ for any positive solution $v$ of $(Q_{\mu, \epsilon})$ with $\mu \in [\overline{\mu}, \overline{\mu}^{-1}]$ and $\epsilon \in (0, 1]$ (cf. \cite[Proposition 6.5]{RQU}).


	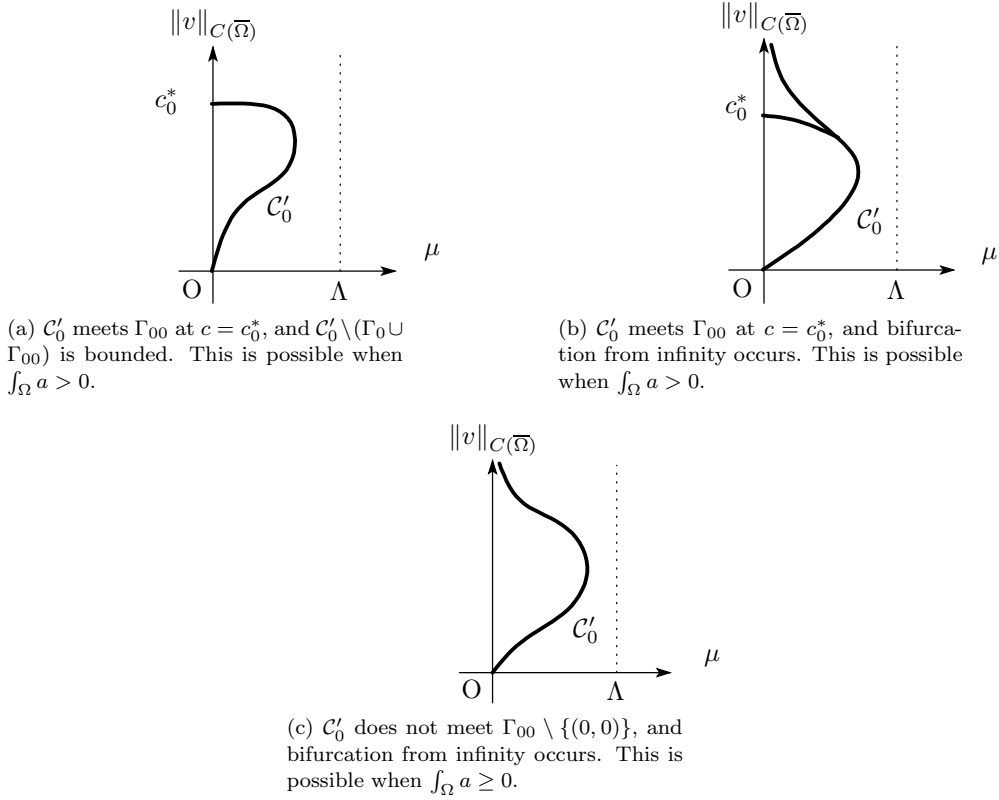
\begin{figure}[!htb]
      \begin{center}
    %
    %
    %
      %
      \subfigure[$\mathcal{C}_0^\prime$ meets $\Gamma_{00}$ at $c=c^*_0$, and
$\mathcal{C}_0^\prime \setminus (\Gamma_0 \cup \Gamma_{00})$ is bounded. This is possible when $\int_\Omega a > 0$.]{
{\unitlength 0.1in
\begin{picture}( 19.8200, 15.3400)( 12.7900,-18.7000)
%
\special{pn 8}%
\special{pa 2132 1702}%
\special{pa 3262 1702}%
\special{fp}%
\special{sh 1}%
\special{pa 3262 1702}%
\special{pa 3194 1682}%
\special{pa 3208 1702}%
\special{pa 3194 1722}%
\special{pa 3262 1702}%
\special{fp}%
%
\special{pn 8}%
\special{pa 2310 1870}%
\special{pa 2310 530}%
\special{fp}%
\special{sh 1}%
\special{pa 2310 530}%
\special{pa 2290 598}%
\special{pa 2310 584}%
\special{pa 2330 598}%
\special{pa 2310 530}%
\special{fp}%
\put(22.0000,-17.9500){\makebox(0,0){O}}%
\put(34.5600,-16.0200){\makebox(0,0){$\mu$}}%
\put(23.1400,-4.1600){\makebox(0,0){$\| v \|_{C(\overline{\Omega})}$}}%
\put(26.6300,-13.7600){\makebox(0,0){$\mathcal{C}_0^\prime$}}%
%
\special{pn 8}%
\special{pa 2976 1702}%
\special{pa 2976 558}%
\special{dt 0.045}%
\put(29.7000,-18.0500){\makebox(0,0){$\Lambda$}}%
%
\special{pn 20}%
\special{pa 2304 1702}%
\special{pa 2322 1636}%
\special{pa 2352 1540}%
\special{pa 2364 1510}%
\special{pa 2378 1480}%
\special{pa 2406 1424}%
\special{pa 2424 1398}%
\special{pa 2444 1374}%
\special{pa 2488 1330}%
\special{pa 2514 1312}%
\special{pa 2542 1292}%
\special{pa 2570 1274}%
\special{pa 2596 1258}%
\special{pa 2624 1238}%
\special{pa 2650 1218}%
\special{pa 2672 1198}%
\special{pa 2694 1174}%
\special{pa 2710 1146}%
\special{pa 2724 1116}%
\special{pa 2734 1084}%
\special{pa 2738 1050}%
\special{pa 2740 1016}%
\special{pa 2736 984}%
\special{pa 2730 952}%
\special{pa 2718 924}%
\special{pa 2702 902}%
\special{pa 2682 882}%
\special{pa 2660 866}%
\special{pa 2634 852}%
\special{pa 2606 842}%
\special{pa 2574 834}%
\special{pa 2506 826}%
\special{pa 2468 824}%
\special{pa 2350 824}%
\special{pa 2310 826}%
\special{pa 2304 826}%
\special{fp}%
\put(20.6700,-8.1200){\makebox(0,0){$c^*_0$}}%
\end{picture}}%
      \label{fig16_0730d}
     }
    \hfill
      \subfigure[$\mathcal{C}_0^\prime$ meets $\Gamma_{00}$ at $c=c^*_0$, and bifurcation from infinity occurs. This is possible when $\int_\Omega a > 0$.]{
{\unitlength 0.1in
\begin{picture}( 20.2400, 15.7200)( 13.1200,-19.2100)
%
\special{pn 8}%
\special{pa 2160 1748}%
\special{pa 3336 1748}%
\special{fp}%
\special{sh 1}%
\special{pa 3336 1748}%
\special{pa 3270 1728}%
\special{pa 3284 1748}%
\special{pa 3270 1768}%
\special{pa 3336 1748}%
\special{fp}%
%
\special{pn 8}%
\special{pa 2344 1922}%
\special{pa 2344 540}%
\special{fp}%
\special{sh 1}%
\special{pa 2344 540}%
\special{pa 2324 606}%
\special{pa 2344 592}%
\special{pa 2364 606}%
\special{pa 2344 540}%
\special{fp}%
\put(22.3000,-18.4300){\makebox(0,0){O}}%
\put(35.2300,-16.7400){\makebox(0,0){$\mu$}}%
\put(23.4700,-4.2900){\makebox(0,0){$\| v \|_{C(\overline{\Omega})}$}}%
\put(29.0100,-14.9200){\makebox(0,0){$\mathcal{C}_0^\prime$}}%
%
\special{pn 8}%
\special{pa 3040 1748}%
\special{pa 3040 568}%
\special{dt 0.045}%
\put(30.3400,-18.5400){\makebox(0,0){$\Lambda$}}%
%
\special{pn 20}%
\special{pa 2338 1746}%
\special{pa 2364 1730}%
\special{pa 2390 1712}%
\special{pa 2414 1696}%
\special{pa 2492 1642}%
\special{pa 2520 1624}%
\special{pa 2546 1604}%
\special{pa 2602 1560}%
\special{pa 2630 1536}%
\special{pa 2658 1510}%
\special{pa 2686 1486}%
\special{pa 2712 1458}%
\special{pa 2738 1432}%
\special{pa 2762 1404}%
\special{pa 2782 1376}%
\special{pa 2800 1348}%
\special{pa 2816 1320}%
\special{pa 2828 1292}%
\special{pa 2836 1264}%
\special{pa 2838 1236}%
\special{pa 2836 1208}%
\special{pa 2830 1182}%
\special{pa 2818 1156}%
\special{pa 2802 1132}%
\special{pa 2780 1106}%
\special{pa 2758 1082}%
\special{pa 2732 1058}%
\special{pa 2706 1036}%
\special{pa 2680 1012}%
\special{pa 2628 968}%
\special{pa 2602 944}%
\special{pa 2554 898}%
\special{pa 2532 876}%
\special{pa 2492 828}%
\special{pa 2472 802}%
\special{pa 2440 750}%
\special{pa 2428 722}%
\special{pa 2416 692}%
\special{pa 2406 662}%
\special{pa 2398 632}%
\special{pa 2382 568}%
\special{fp}%
\put(22.0200,-8.9400){\makebox(0,0){$c^*_0$}}%
%
\special{pn 20}%
\special{pa 2732 1054}%
\special{pa 2674 1026}%
\special{pa 2644 1014}%
\special{pa 2614 1000}%
\special{pa 2584 988}%
\special{pa 2524 968}%
\special{pa 2494 960}%
\special{pa 2430 948}%
\special{pa 2400 944}%
\special{pa 2368 942}%
\special{pa 2338 938}%
\special{fp}%
\end{picture}}%
        \label{fig17_0309}
      }
\hfill
      \subfigure[$\mathcal{C}_0^\prime$ does not meet $\Gamma_{00}\setminus \{ (0,0)\}$, and bifurcation from infinity occurs. This is possible when $\int_\Omega a \geq 0$.]{
{\unitlength 0.1in
\begin{picture}( 19.8100, 14.5700)( 12.5900,-17.9000)
%
\special{pn 8}%
\special{pa 2116 1632}%
\special{pa 3218 1632}%
\special{fp}%
\special{sh 1}%
\special{pa 3218 1632}%
\special{pa 3152 1612}%
\special{pa 3166 1632}%
\special{pa 3152 1652}%
\special{pa 3218 1632}%
\special{fp}%
%
\special{pn 8}%
\special{pa 2290 1790}%
\special{pa 2290 522}%
\special{fp}%
\special{sh 1}%
\special{pa 2290 522}%
\special{pa 2270 588}%
\special{pa 2290 574}%
\special{pa 2310 588}%
\special{pa 2290 522}%
\special{fp}%
\put(21.8200,-17.1900){\makebox(0,0){O}}%
\put(34.4000,-15.5600){\makebox(0,0){$\mu$}}%
\put(22.9400,-4.1300){\makebox(0,0){$\| v \|_{C(\overline{\Omega})}$}}%
%
\special{pn 20}%
\special{pa 2290 1632}%
\special{pa 2310 1606}%
\special{pa 2330 1582}%
\special{pa 2372 1534}%
\special{pa 2396 1510}%
\special{pa 2418 1490}%
\special{pa 2470 1450}%
\special{pa 2496 1432}%
\special{pa 2522 1416}%
\special{pa 2550 1398}%
\special{pa 2604 1364}%
\special{pa 2656 1326}%
\special{pa 2680 1306}%
\special{pa 2704 1282}%
\special{pa 2724 1256}%
\special{pa 2742 1230}%
\special{pa 2758 1200}%
\special{pa 2772 1170}%
\special{pa 2780 1140}%
\special{pa 2786 1108}%
\special{pa 2786 1076}%
\special{pa 2782 1044}%
\special{pa 2774 1012}%
\special{pa 2762 982}%
\special{pa 2746 952}%
\special{pa 2726 924}%
\special{pa 2706 900}%
\special{pa 2680 878}%
\special{pa 2654 858}%
\special{pa 2626 840}%
\special{pa 2598 824}%
\special{pa 2568 810}%
\special{pa 2510 780}%
\special{pa 2482 766}%
\special{pa 2456 748}%
\special{pa 2434 728}%
\special{pa 2412 706}%
\special{pa 2394 682}%
\special{pa 2378 656}%
\special{pa 2362 626}%
\special{pa 2350 596}%
\special{pa 2336 566}%
\special{pa 2326 536}%
\special{fp}%
\put(27.8000,-14.1000){\makebox(0,0){$\mathcal{C}_0^\prime$}}%
%
\special{pn 8}%
\special{pa 2940 1632}%
\special{pa 2940 546}%
\special{dt 0.045}%
\put(29.3400,-17.2900){\makebox(0,0){$\Lambda$}}%
\end{picture}}%
          \label{fig16_0730c}
      }
      \end{center}
      \caption{Possible bifurcation diagrams for $\mathcal{C}_0^\prime$ when $\int_\Omega a \geq 0$.}
      \label{fig16_0730cd}
    \end{figure}
We conclude now the proof of Theorem \ref{tp}. 
By the rescaling $u = \lambda^{\frac{1}{p-q}}v$, we transform the component $\mathcal{C}_0^\prime$ for $(Q_\mu)$ into a component of non-negative solutions for $(P_\lambda)$. By \cite[Lemma 6.8(1)]{RQU} we know that $(P_\lambda)$ has no nontrivial non-negative solutions for $\lambda=0$ (assertions (ii) and (iii) are thus verified). From \cite[Proposition 4.2]{RQUTMNA} the following \textit{a priori} upper bound for non-negative solutions of  $(P_{\lambda})$ holds: given $\overline{\lambda} > 0$, there exists $C_{\overline{\lambda}} > 0$ such that $u \leq C_{\overline{\lambda}}$ on $\overline{\Omega}$ for all non-negative solutions of $(P_\lambda)$ with $\lambda \in [-\overline{\lambda}, \overline{\lambda}]$.   Combining (c1) and (c2) with these two assertions provides us with  the desired component satisfying properties (i) and (iv).
The proof of Theorem \ref{tp} is now complete. \qed \newline

\begin{remark} \label{rem:thm1}
When $\int_\Omega b = 0<\int_\Omega a$, we may consider,
instead of $(Q_{\mu, \epsilon})$, the problem
\begin{align} \label{Qme2}
\begin{cases}
-\Delta v = \mu \left(
(b(x)-\epsilon) (v+\epsilon)^{q-2}v + a(x) v^{p-1} \right) & \mbox{in $\Omega$}, \\
\dfrac{\partial v}{\partial \mathbf{n}} = 0 & \mbox{on $\partial \Omega$},
\end{cases}
\end{align}
where $b(x) - \epsilon$ changes sign if $\epsilon > 0$ is small enough.
In fact, the eigenvalue problem associated with \eqref{Qme2}, as introduced in \eqref{160729epr}, possesses exactly two principal eigenvalues $0, \mu_\epsilon$,
with $\mu_\epsilon > 0$, and $\mu_\epsilon \rightarrow 0$ as $\epsilon \to 0^+$ (see \cite[Lemma 6.6]{RQU}). However, since $c^*_\epsilon \to 0$ as $\epsilon \to 0^+$, we can not exclude the possibility that $\mathcal{C}_\epsilon$ shrinks to $\{ (0,0) \}$ as $\epsilon \to 0^+$ when $\mathcal{C}_\epsilon \setminus (\Gamma_0 \cup \Gamma_{00})$ is bounded, see Figure \ref{fig16_0730b}.

Finally, in the case $\int_\Omega a = \int_\Omega b = 0$, $\mathcal{C}_\epsilon$ is provided from \eqref{Qme2} as in Figure \ref{fig16_0730aa}. Indeed, letting $v_n$ be a positive solution of \eqref{Qme2} for $\mu = \mu_n > 0$, it is not possible that $(\mu_n, v_n) \rightarrow (0, c)$ in $\mathbb{R}\times C(\overline{\Omega})$ for some constant $c\geq 0$.
However, we can not exclude the possibility that $\mathcal{C}_{0, \rho} \subset \Gamma_0 \cup \Gamma_{00}$, since \eqref{Mfrml} holds for any positive constant $M$ in this case. \newline
\end{remark}
\noindent\textit{Note added in proof.}
(i) Regarding $(H_3)$, the condition that $\Omega_-^b$ is a subdomain can be removed from Theorem \ref{tp}. Indeed, although this condition is needed to verify the non-existence of nontrivial non-negative solutions of $(P_\lambda)$ bifurcating from $\{ (\lambda, 0) \}$ for $\lambda < 0$, such verification is required only for $\lambda > 0$ in Theorem \ref{tp}.

(ii) Let $\overline{\mathcal{C}_0}$ be a maximal component of nonnegative solutions of $(P_\lambda)$ that includes the loop type component $\mathcal{C}_0$ provided by Theorem \ref{tp} and such that $\overline{\mathcal{C}_0}\setminus \{ (0,0) \}$ consists of nontrivial non-negative solutions. As a further result for Theorem \ref{tp}, we obtain that $\overline{\mathcal{C}_0}$ is bounded in $\mathbb{R}\times C(\overline{\Omega})$ if, in addition to the hypotheses in Theorem \ref{tp}, one of the following conditions is assumed.
\begin{enumerate}
\item[(a)] $\overline{\Omega_+^b}\subset \Omega$, $\Omega^{\prime\prime} := \Omega \setminus \overline{\Omega_+^b}$ is a subdomain, and $\Omega_+^a \subset \Omega_+^b$.
\item[(b)] $\Omega_+^b$ contains a tubular neighborhood of $\partial \Omega$, $\Omega^{\prime\prime}$ is a subdomain, and $\Omega_+^a \subset \Omega_+^b$.
\end{enumerate}
Indeed, under the additional condition, the strong maximum principle and boundary point lemma show that any nontrivial non-negative solution of $(Q_{\mu})$ is positive in $\Omega_+^b$. Consequently, Lemma 6 (i) is also valid for nontrivial non-negative solutions of $(Q_{\mu})$, and the desired conclusion follows.

\section{A further analysis for the case $\int_\Omega a < 0$} \label{sec:a<0}
Let us assume now condition \eqref{prehypo}. In addition, we assume $(H_0)$, $(H_3)$, and condition (b) from Theorem \ref{tp} with $\Omega' = \Omega^a_- \neq \emptyset$.
Then  $(P_\lambda)$ has a bounded loop type component of non-negative solutions $\mathcal{C}_0$ in $\mathbb{R} \times C(\overline{\Omega})$,
satisfying the following properties (see \cite[Theorem 1.6]{RQU} and Figure
\ref{fig16_0730f}):
\begin{enumerate}
\item $\mathcal{C}_0$ bifurcates at $(0,0)$ and joins $(0,0)$ to itself;
\item $\mathcal{C}_0$ is non-trivial, i.e., $\mathcal{C}_0 \neq \{(0,0)\}$.
More precisely, $\mathcal{C}_0$ contains a positive solution $u_0$ of $(P_\lambda)$ with $\lambda = 0$;
\item  The only trivial solution contained in $\mathcal{C}_0$ is $(\lambda, u)=(0,0)$, i.e., $\mathcal{C}_0$ does not contain any $(\lambda,0)$ with $\lambda \neq 0$.
\item There exists $\delta > 0$ such that $\mathcal{C}_0$ does not contain any positive solution $u$ of $(P_\lambda)$ with $\lambda = 0$ satisfying $\Vert u \Vert_{C(\overline{\Omega})} \leq \delta$.
\end{enumerate}

According to the arguments developed in \cite{RQU}, this existence result can be verified by considering the regularized version of $(P_\lambda)$ for $u^{q-1}$ at $u=0$:
$$
\begin{cases}
-\Delta u =  \lambda (b(x) - \epsilon) (u+\epsilon)^{q-2}u +a(x)u^{p-1} & \mbox{in} \ \Omega, \\
\dfrac{\partial u}{\partial \mathbf{n}} = 0 & \mbox{on} \ \partial \Omega,
\end{cases} \leqno{(P_{\lambda,\epsilon})}
$$
where $\epsilon > 0$ and $\Omega^{b-\epsilon}_+ \not= \emptyset$. Then $(P_{\lambda, \epsilon})$ is regular, so that the unilateral global bifurcation theorem by L\'opez-G\'omez \cite[Theorem 6.4.3]{LG01} may be applied. To this end, we consider the linearized problem at $u=0$:
\begin{align*} 
\begin{cases}
-\Delta \varphi = \lambda (b-\epsilon) \epsilon^{q-2} \varphi & \mbox{in} \ \Omega, \\
\dfrac{\partial \varphi}{\partial \mathbf{n}} = 0 & \mbox{on} \
\partial \Omega.
\end{cases}
\end{align*}
Since
$b-\epsilon$ changes sign and $\int_\Omega (b-\epsilon) < 0$, this eigenvalue problem has exactly two principal eigenvalues, $ \lambda=0$ and $\lambda=\lambda_\epsilon>0$, which are both simple.
We use now the unilateral global bifurcation theory to obtain two components $\mathcal{C}_{0, \epsilon}$ and $\mathcal{C}_{1, \epsilon}$ of positive solutions of $(P_{\lambda,\epsilon})$, bifurcating from $(0,0)$ and $(\lambda_\epsilon,0)$, respectively. Moreover, we can analyze the local nature of these  components at the bifurcation points by using the local bifurcation theory proposed by Crandall and Rabinowitz. We can also analyze the global nature of these components by making good use of an {\it a priori} bound in $\mathbb{R} \times C(\overline{\Omega})$ for positive solutions $(\lambda, u)$ of $(P_{\lambda, \epsilon})$. Consequently, $\mathcal{C}_{0, \epsilon}$ and $\mathcal{C}_{1, \epsilon}$ are both bounded, so that $\mathcal{C}_{0, \epsilon} = \mathcal{C}_{1, \epsilon} (=: \mathcal{C}_\epsilon)$,
i.e., $\mathcal{C}_\epsilon$ is a mushroom. Finally, based on the fact that $\lambda_\epsilon \to 0$ as $\epsilon \to 0^+$
(see \cite[Lemma 6.6]{RQU}), we may apply Whyburn's topological method to infer that $\mathcal{C}_0 = \limsup_{\epsilon \to 0^+} \mathcal{C}_\epsilon$ is  a non-empty component of non-negative solutions $(\lambda, u)$ of $(P_\lambda)$. The limiting features of $\mathcal{C}_0$ mentioned above follow by using some additional results on the set of non-negative solution of $(P_\lambda)$.

In addition to these features, we shall provide a further result on the direction of bifurcation at $(0,0)$ for $\mathcal{C}_0$, using  Whyburn's topological method again. We remark that, although properties (i)-(iv) above provide that
$\mathcal{C}_0$ is a 	
loop, i.e.,
$\mathcal{C}_0$ joins $(0,0)$ to itself passing by $(0, u_0)$ for some positive solution $u_0$ of $(P_\lambda)$ with $\lambda = 0$, Theorem \ref{thm:dir} confirms additionally 		
the loop property of $\mathcal{C}_0$, as follows:

Given $\rho > 0$, we set
\begin{align*}
B_\rho((\lambda_1, u_1)) = \{ (\lambda, u) \in \mathbb{R} \times C(\overline{\Omega}) : |\lambda - \lambda_1| + \| u - u_1 \|_{C(\overline{\Omega})} < \rho \}. 
\end{align*}

\begin{theorem} \label{thm:dir}
Under the conditions stated above,
$\mathcal{C}_{0}$ contains closed connected sets $\mathcal{C}^{\pm}_{0}$ such that $(0,0) \in\mathcal{C}^{\pm}_{0}$ and $\mathcal{C}^{\pm}_0 \neq \{ (0,0)\}$. Moreover, if $(\lambda, u) \in\mathcal{C}^{\pm}_{0} \setminus\{ (0,0)\}$ then $\lambda\gtrless0$, i.e. $\mathcal{C}%
_{0}$ bifurcates both to the left and to the right at $(0,0)$,  see Figure \ref{fig16_0709}.
\end{theorem}

	 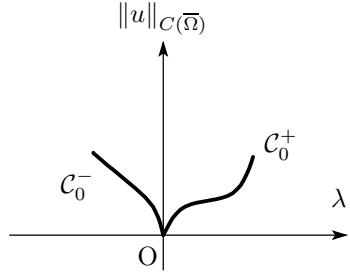
\begin{figure}[!htb]
  	   \begin{center}
{\unitlength 0.1in
\begin{picture}( 20.0100, 14.0000)( 18.0700,-22.5000)
%
\special{pn 8}%
\special{pa 2048 2068}%
\special{pa 3808 2068}%
\special{fp}%
\special{sh 1}%
\special{pa 3808 2068}%
\special{pa 3742 2048}%
\special{pa 3756 2068}%
\special{pa 3742 2088}%
\special{pa 3808 2068}%
\special{fp}%
%
\special{pn 8}%
\special{pa 2852 2250}%
\special{pa 2852 1060}%
\special{fp}%
\special{sh 1}%
\special{pa 2852 1060}%
\special{pa 2832 1126}%
\special{pa 2852 1112}%
\special{pa 2872 1126}%
\special{pa 2852 1060}%
\special{fp}%
\put(37.7100,-18.7000){\makebox(0,0){$\lambda$}}%
\put(28.4700,-9.3000){\makebox(0,0){$\| u \|_{C(\overline{\Omega})}$}}%
\put(27.7100,-21.7700){\makebox(0,0){O}}%
%
\special{pn 20}%
\special{pa 2852 2068}%
\special{pa 2884 2008}%
\special{pa 2900 1980}%
\special{pa 2920 1956}%
\special{pa 2942 1934}%
\special{pa 2968 1918}%
\special{pa 2998 1906}%
\special{pa 3028 1898}%
\special{pa 3062 1892}%
\special{pa 3094 1886}%
\special{pa 3128 1880}%
\special{pa 3160 1872}%
\special{pa 3190 1862}%
\special{pa 3216 1848}%
\special{pa 3238 1828}%
\special{pa 3258 1806}%
\special{pa 3274 1778}%
\special{pa 3288 1750}%
\special{pa 3302 1718}%
\special{pa 3314 1684}%
\special{pa 3322 1658}%
\special{fp}%
%
\special{pn 20}%
\special{pa 2852 2068}%
\special{pa 2844 2036}%
\special{pa 2836 2006}%
\special{pa 2826 1976}%
\special{pa 2812 1946}%
\special{pa 2798 1918}%
\special{pa 2758 1870}%
\special{pa 2710 1826}%
\special{pa 2684 1804}%
\special{pa 2660 1784}%
\special{pa 2610 1742}%
\special{pa 2586 1722}%
\special{pa 2560 1702}%
\special{pa 2512 1660}%
\special{pa 2488 1640}%
\special{pa 2486 1636}%
\special{fp}%
\put(34.6800,-15.9200){\makebox(0,0){$\mathcal{C}^+_0$}}%
\put(24.0400,-18.0400){\makebox(0,0){$\mathcal{C}^-_0$}}%
\end{picture}}%
	  \caption{A bifurcation diagram for $\mathcal{C}_0$ at $(0,0)$: the case $\int_\Omega a < 0$.}
	\label{fig16_0709}
	  \end{center}
	    \end{figure}

\begin{proof}
Let $\Sigma^{+}_{\epsilon}$ and $\Sigma^{-}_{\epsilon}$ be closed connected subsets of $\{ (\lambda, u) \in \mathcal{C}_\epsilon : \lambda \geq 0 \}$ and $\{ (\lambda, u) \in \mathcal{C}_\epsilon : \lambda \leq 0 \}$, respectively, such that $(0,0), (0, u_\epsilon^-) \in \Sigma^-_\epsilon$, and
$(\lambda_\epsilon, 0), (0, u_\epsilon^+)$  $\in \Sigma^+_{\epsilon}$ for some positive solutions $u_\epsilon^{\pm}$ of $(P_\lambda)$ with $\lambda = 0$, see Figure  \ref{fig17_0309c}.
%
	 \begin{figure}[!htb]
	   \begin{center}
{\unitlength 0.1in
\begin{picture}( 22.3400, 17.0000)( 19.6800,-21.5000)
%
\special{pn 8}%
\special{pa 2208 1936}%
\special{pa 4202 1936}%
\special{fp}%
\special{sh 1}%
\special{pa 4202 1936}%
\special{pa 4136 1916}%
\special{pa 4150 1936}%
\special{pa 4136 1956}%
\special{pa 4202 1936}%
\special{fp}%
%
\special{pn 8}%
\special{pa 3120 2150}%
\special{pa 3120 740}%
\special{fp}%
\special{sh 1}%
\special{pa 3120 740}%
\special{pa 3100 806}%
\special{pa 3120 792}%
\special{pa 3140 806}%
\special{pa 3120 740}%
\special{fp}%
\put(42.4300,-18.1000){\makebox(0,0){$\lambda$}}%
\put(31.1000,-5.3000){\makebox(0,0){$\| u\|_{C(\overline{\Omega})}$}}%
\put(30.4300,-20.0900){\makebox(0,0){O}}%
\put(26.6300,-13.9900){\makebox(0,0){$\Sigma_{\epsilon}^-$}}%
\put(40.0000,-14.0000){\makebox(0,0){$\Sigma_{\epsilon}^+$}}%
\put(35.4000,-20.4000){\makebox(0,0){$\lambda_\epsilon$}}%
%
\special{pn 20}%
\special{pa 3508 1930}%
\special{pa 3558 1884}%
\special{pa 3606 1838}%
\special{pa 3628 1814}%
\special{pa 3652 1790}%
\special{pa 3672 1766}%
\special{pa 3692 1740}%
\special{pa 3712 1716}%
\special{pa 3728 1688}%
\special{pa 3744 1662}%
\special{pa 3758 1634}%
\special{pa 3770 1604}%
\special{pa 3780 1574}%
\special{pa 3792 1510}%
\special{pa 3794 1478}%
\special{pa 3796 1444}%
\special{pa 3794 1410}%
\special{pa 3790 1376}%
\special{pa 3784 1342}%
\special{pa 3776 1308}%
\special{pa 3766 1276}%
\special{pa 3756 1242}%
\special{pa 3728 1182}%
\special{pa 3710 1152}%
\special{pa 3692 1124}%
\special{pa 3672 1100}%
\special{pa 3650 1076}%
\special{pa 3628 1054}%
\special{pa 3604 1036}%
\special{pa 3578 1018}%
\special{pa 3552 1002}%
\special{pa 3496 978}%
\special{pa 3466 968}%
\special{pa 3434 958}%
\special{pa 3404 950}%
\special{pa 3370 944}%
\special{pa 3338 938}%
\special{pa 3270 930}%
\special{pa 3202 924}%
\special{pa 3168 922}%
\special{pa 3132 920}%
\special{pa 3118 918}%
\special{fp}%
%
\special{pn 8}%
\special{pa 3118 918}%
\special{pa 3078 928}%
\special{pa 3040 938}%
\special{pa 3012 948}%
\special{pa 2998 958}%
\special{pa 2998 970}%
\special{pa 3014 982}%
\special{pa 3038 996}%
\special{pa 3072 1014}%
\special{pa 3108 1032}%
\special{pa 3146 1056}%
\special{pa 3180 1082}%
\special{pa 3210 1110}%
\special{pa 3228 1142}%
\special{pa 3238 1172}%
\special{pa 3230 1196}%
\special{pa 3210 1212}%
\special{pa 3180 1224}%
\special{pa 3142 1230}%
\special{pa 3118 1234}%
\special{fp}%
\put(33.7000,-13.1000){\makebox(0,0){$(0, u_\epsilon^-)$}}%
\put(34.3000,-8.2000){\makebox(0,0){$(0, u_\epsilon^+)$}}%
%
\special{pn 20}%
\special{pa 3118 1932}%
\special{pa 3098 1864}%
\special{pa 3086 1832}%
\special{pa 3072 1804}%
\special{pa 3054 1776}%
\special{pa 3034 1754}%
\special{pa 3008 1738}%
\special{pa 2978 1726}%
\special{pa 2944 1718}%
\special{pa 2874 1708}%
\special{pa 2842 1702}%
\special{pa 2812 1694}%
\special{pa 2786 1682}%
\special{pa 2768 1664}%
\special{pa 2756 1640}%
\special{pa 2752 1612}%
\special{pa 2754 1578}%
\special{pa 2760 1544}%
\special{pa 2772 1508}%
\special{pa 2786 1472}%
\special{pa 2804 1438}%
\special{pa 2826 1406}%
\special{pa 2870 1354}%
\special{pa 2922 1314}%
\special{pa 2950 1298}%
\special{pa 2978 1284}%
\special{pa 3008 1270}%
\special{pa 3036 1260}%
\special{pa 3066 1248}%
\special{pa 3098 1238}%
\special{pa 3118 1232}%
\special{fp}%
\end{picture}}%
	  \caption{The behaviors of $\Sigma_{\epsilon}^\pm$.}
	\label{fig17_0309c}
	  \end{center}
	    \end{figure}
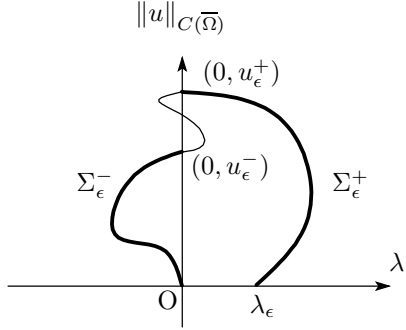
$\Sigma^{+}_{\epsilon}$ and $\Sigma^{-}_{\epsilon}$ are well defined by virtue of the following facts, see \cite[Sections 5 and 6]{RQU}.

\begin{itemize}
\item $\mathcal{C}_\epsilon = \{ (\lambda, u) \}$ is continuously parametrized by $\lambda = \lambda_k (s), u=u_k (s)$ for $s \in [0, s_0)$, $k=0,1$, in certain neighborhoods of the bifurcation points $(0,0)$ and $(\lambda_\epsilon, 0)$. In addition, $\lambda_k$ and $u_k$ satisfy $(\lambda_0(0), u_0(0)) = (0,0), (\lambda_1(0), u_1(0))=(\lambda_\epsilon, 0)$, respectively;
\item $\mathcal{C}_\epsilon$ bifurcates to the region $\lambda < 0$ at $(0,0)$ under condition \eqref{prehypo};
\item $\mathcal{C}_\epsilon$ contains a positive solution $w_0$ of $(P_\lambda)$ with $\lambda = 0$;
\item $\mathcal{C}_\epsilon$ does not contain any point $(\lambda,0)$ with $\lambda \neq 0, \lambda_\epsilon$;
\item There exists $\delta > 0$, independent of $\epsilon$, such that $\mathcal{C}_\epsilon$ does not contain any positive solution $u$ of $(P_{\lambda, \epsilon})$ with $\lambda = 0$ satisfying $\Vert u \Vert_{C(\overline{\Omega})} \leq \delta$. Consequently, for the positive solution $w_0$ above, we have that $\| w_0 \|_{C(\overline{\Omega})}$ is bounded below by some positive constant independent of $\epsilon$.
\end{itemize}

Since $\Sigma^\pm_{\epsilon} \subset \mathcal{C}_\epsilon$, we observe that
\begin{align*} 
\Sigma^\pm_{0} :=
\limsup_{\epsilon \to 0^+} \Sigma^\pm_{\epsilon} \subset
\limsup_{\epsilon \to 0^+} \mathcal{C}_\epsilon = \mathcal{C}_{0}.
\end{align*}
Repeating the argument above,
Whyburn's topological approach yields that
$\Sigma^\pm_{0}$ are non-empty, closed and connected sets consisting of non-negative solutions of $(P_\lambda)$ and such that $(0,0) \in \liminf_{\epsilon \to 0^+} \Sigma^\pm_{\epsilon} \subset \Sigma^\pm_{0}$. By property (iv), we have $(0, u_0^{\pm}) \in \Sigma_{0}^{\pm}$ for some positive solutions $u_0^{\pm}$ of $(P_\lambda)$ with $\lambda = 0$. It follows that
$\Sigma^\pm_{0} \neq \{ (0,0) \}$, and moreover, by property (iii), that $\Sigma^\pm_{0} \setminus \{ (0,0)\}$ consists of nontrivial non-negative solutions of $(P_\lambda)$.

Now, by definition, we see that $(\lambda, u) \in \Sigma^+_{0}$ (respect.\ $\Sigma^-_{0}$) implies $\lambda \geq 0$ (respect.\ $\lambda \leq 0$).
Lastly, by using property (iv) again, there exists $\rho > 0$ small enough such that  $\Sigma_{0, \rho}^{\pm} := \Sigma^\pm_{0} \cap \overline{B_\rho ((0,0))}$
is closed and connected, and
if $(\lambda, u) \in \Sigma^\pm_{0, \rho} \setminus \{ (0,0) \}$
then $\lambda \gtrless 0$. Therefore $\mathcal{C}^\pm_0 := \Sigma^\pm_{0, \rho}$ have the desired properties.
\end{proof}

\vspace{1cm}

\end{document}